\documentclass[11pt]{amsart}

\usepackage{hyperref}

\usepackage{graphicx, enumerate, url}
\usepackage{amssymb,mathtools}
\usepackage{algorithm}
\usepackage{algpseudocode}
\usepackage{color}
\usepackage{tikz}
\usetikzlibrary{3d,calc}
\usepackage[title]{appendix}

\numberwithin{equation}{section}
\numberwithin{algorithm}{section}

\theoremstyle{plain}
\newtheorem{theorem}{Theorem}[section]
\newtheorem{proposition}[theorem]{Proposition}
\newtheorem{lemma}[theorem]{Lemma}

\theoremstyle{definition}

\newtheorem{example}[theorem]{Example}
\newtheorem{problem}[theorem]{Problem}

\theoremstyle{remark}

\let\H\undefined
\let\S\undefined

\DeclareMathOperator{\spa}{span}

\DeclareMathOperator{\tr}{Tr}

\DeclareMathOperator{\rank}{rank}

\DeclareMathOperator{\H}{H}
\DeclareMathOperator{\S}{S}

\DeclareMathOperator{\supp}{supp}

\newcommand{\N}{\mathbb{N}}
\newcommand{\C}{\mathbb{C}}
\newcommand{\F}{\mathbb{F}}
\newcommand{\R}{\mathbb{R}}
\newcommand{\cB}{\mathcal{B}}
\newcommand{\cH}{\mathcal{H}}
\newcommand{\cM}{\mathcal{M}}

\newcommand{\cX}{\mathcal{X}}

\newcommand{\ba}{\mathbf{a}}
\newcommand{\bb}{\mathbf{b}}
\newcommand{\bc}{\mathbf{c}}
\newcommand{\be}{\mathbf{e}}
\newcommand{\bbf}{\mathbf{f}}
\newcommand{\bg}{\mathbf{g}}
\newcommand{\bu}{\mathbf{u}}
\newcommand{\bU}{\mathbf{U}}
\newcommand{\bv}{\mathbf{v}}
\newcommand{\bV}{\mathbf{V}}
\newcommand{\bW}{\mathbf{W}}
\newcommand{\bx}{\mathbf{x}}
\newcommand{\bX}{\mathbf{X}}
\newcommand{\by}{\mathbf{y}}
\newcommand{\bY}{\mathbf{Y}}
\newcommand{\bz}{\mathbf{z}}
\newcommand{\0}{\mathbf{0}}
\newcommand{\1}{\mathbf{1}}
\newcommand{\rB}{\mathrm{B}}
\newcommand{\rK}{\mathrm{K}}
\newcommand{\rS}{\mathrm{S}}
\newcommand{\rT}{\mathrm{T}}
\begin{document}
\title{On quantum Strassen's theorem}
\author[Shmuel Friedland]{Shmuel~Friedland}
\address{Department of Mathematics and Computer Science, University of Illinois at Chicago, Chicago, Illinois, 60607-7045, USA }
\email{friedlan@uic.edu}
\author[Jingtong Ge]{Jingtong~Ge}
\address{University of Chinese Academy of Sciences, Beijing 100049, China, and
University of Technology Sydney, NSW, Australia}
\email{gejingtong@amss.ac.cn}
\author[Lihong Zhi]{Lihong~Zhi}
\address{KLMM, Academy of Mathematics and Systems Science, Chinese Academy of Sciences, Beijing 100190, China, University of Chinese Academy of Sciences, Beijing 100049, China }
\email{lzhi@mmrc.iss.ac.cn}
\date{July 24, 2020}
\begin{abstract} Strassen's theorem circa 1965 gives necessary and sufficient conditions on  the existence of a probability measure on two product spaces with  given support and two marginals.  In the case where each product space is finite Strassen's theorem is reduced to a linear programming problem which can be solved using flow theory.
A density matrix of bipartite quantum system is a quantum analog of a probability matrix on two finite product spaces.  Partial traces of the density matrix are analogs of marginals. The support of the density matrix is its range.  The analog of Strassen's theorem in this case can be stated and solved using semidefinite programming.  The aim of this paper is to give analogs of Strassen's theorem to density trace class operators on a  product of two separable Hilbert spaces, where at least one of the Hilbert spaces is infinite dimensional.
\end{abstract}

\begin{abstract} Strassen's theorem circa 1965 gives necessary and sufficient conditions on  the existence of a probability measure on two product spaces with  given support and two marginals.  In the case where each product space is finite Strassen's theorem is reduced to a linear programming problem which can be solved using flow theory.
A density matrix of bipartite quantum system is a quantum analog of a probability matrix on two finite product spaces.  Partial traces of the density matrix are analogs of marginals. The support of the density matrix is its range.  The analog of Strassen's theorem in this case can be stated and solved using semidefinite programming.  The aim of this paper is to give analogs of Strassen's theorem to density trace class operators on a  product of two separable Hilbert spaces, where at least one of the Hilbert spaces is infinite dimensional.
\end{abstract}

\maketitle


\section{Introduction}
Let $\mu$ be a probability measure on the discrete space $\Omega=[m]\times [n]$, where $m,n$ are positive integers and $[m]=\{1,\ldots,m\}$.  A probability measure on $\Omega$ is a nonnegative $m\times n$ matrix $A=[a_{ij}]\in\R_+^{m\times n}$, such that the sum of its entires is $1$.  Let $\1_m=(1,\ldots,1)^\top\in \R^m$.  Then the marginals: $\mu_1=A\1_n$ and $\mu_2=A^\top \1_m$ are  the probability measures on $[m]$ and $[n]$ respectively. The support of $\mu$, denoted as $\supp \mu$, is the following bipartite graph $G=(V,E)$, where $V=[m]\cup [n]$ and $E=\{(i,j), i\in[m], j\in[n], a_{ij}>0\}$.  The following inverse problem is natural:
\begin{problem}\label{margprob}
Given probability measures $\mu_1$ and $\mu_2$ on $[m]$ and $[n]$ respectively,
find necessary and sufficient conditions for existence of a probability measure $\mu$ on $[m]\times [n]$, whose support is contained in a given bipartite graph $G=([m]\cup [n],E)$ and whose marginals are $\mu_1$ and $\mu_2$.
\end{problem}
This problem is a classical problem in combinatorial optimization \cite{CCPS}, and can be solved using the standard flow theory \cite{FF62}.  See \cite{Hsu17}.   Strassen \cite{Str65} gave a solution of Problem \ref{margprob} to a measure on the Borel $\sigma$-algebra of the product of two compact metric spaces. (Strassen did not bother to state the finite space case.  Actually, Strassen considered a more general $\varepsilon\ge 0$ version of Problem \ref{margprob}.)

In a recent paper \cite{Ming}, Zhou et al.  gave necessary and sufficient conditions for the analog of Problem \ref{margprob} in the quantum setting: Let $\cH$ be a finite dimensional inner product space of dimension $n$ over the complex numbers $\C$.  We identify $\cH$ with $\C^n$ with the standard inner product $\langle \bx,\by\rangle=\by^*\bx$.
Then $\rB(\cH)$ the set of linear operators $A:\cH\to\cH$ is the algebra of $n\times n$ matrices $\C^{n\times n}$.  The set of selfadjoint operators $\rS(\cH)$ is the real space of $n\times n$ Hermitian matrices. Then $\rS_+(\cH)\supset \rS_{+,1}(\cH)$ is the cone of positive semidefinite matrices in $\rS(\cH)$ and the convex set of positive semidefinite matrices of trace $1$.  The set $\rS_{+,1}(\cH)$, which is called the space of density matrices, is the analog of the set of probability measure in quantum physics.  (It is also called the space of mixed states \cite{NC00}.)  On $\rS(\cH)$ we have a partial order $A\succeq B$ if $A-B\in \rS_+(\cH)$.
For $\rho\in \rS(\cH)$ the support of $\rho$, denoted as $\supp \rho$, is $\rho(\cH)$, i.e. the subspace spanned by the nonzero eigenvectors of $\rho$, which is the range of $\rho$.

Let $\cH_1\equiv \C^m, \cH_2\equiv \C^n$.  Then $\cH=\cH_1\otimes \cH_2\equiv \C^n\otimes \C^n\equiv \C^{m\times n}$ is called the bipartite space.  The space $\rB(\cH)$ can be viewed as $(mn)\times (mn)$ matrices $T=[t_{(i,p)(j,q)}]\in \C^{(mn)\times(mn)}$, where $i,j\in[m], p,q\in[n]$.  There are two natural contraction maps, which are called partial traces:
\begin{eqnarray*}
\tr_2 :\rB(\cH)\to \rB(\cH_1) ;\tr_2 T=\left[\sum_{p=1}^n t_{(i,p)(j,p)}\right], i,j\in [m],\\
\tr_1 :\rB(\cH)\to\rB(\cH_2) ;\tr_1 T=\left[\sum_{i=1}^m t_{(i,p)(i,q)}\right], p,q\in [n].
\end{eqnarray*}

A density matrix $\rho\in \rS_{+,1}(\cH)$ is an analog of a probability measure $\mu$ on $[m]\times [n]$.
Clearly $\rho_1=\tr_2\rho\in \rS_{+,1}(\cH_1)$ and $\rho_2=\tr_1\rho\in \rS_{+,1}(\cH_2)$
are the analogs of marginals $\mu_1$ and $\mu_2$.  
Hence the analog of Problem \ref{margprob} is the quantum marginals and  coupling  problems  \cite{BRAVYI2003,CLL13,CJK2014,YZYY2018}.
\begin{problem}\label{margproddm}  Let $\cH=\cH_1\otimes\cH_2$, where $\cH_1$ and $\cH_2$ are finite dimensional inner product spaces.
Let $\mathcal{X}\subseteq \mathcal{H}$ be a closed subspace. Given $\rho_{i}\in \rS_{+,1}(\H_i)$, $i=1,2$,  what are necessary and sufficient conditions for the existence of $\rho \in \rS_{+,1}(\cH)$, $\supp \rho\subseteq \cX$, such that $\rho_{1},\rho_{2}$ are its partial traces?
\end{problem}

This problem is a variation of the classical $2$-representability problem: Given $\rho_i\in\rS_{+,1}(\cH_i)$ for $i\in[2]$, does there exists a pure state $\rho\in\rS_{+,1}(\cH_1\otimes \cH_2)$ such that $\rho_1,\rho_2$ are its partial traces?
This problem was solved by Klyachko in \cite{Kly04} for a more general case: Namely the spectrum of the mixed bipartite state is prescribed.  The $N$-representation problem \cite{Cou60,Col63, Kly06} was considered to be in mid 90's one of ten most
prominent research challenges in quantum chemistry \cite{Sti}.
Recently, some aspects of the three partite quantum marginals problem were discussed in \cite{CLL13}.

This problem can be stated in terms of semidefinite problem (SDP): Let $P_{\cX}$ be the projection on the $\cX$. Consider the maximum problem
\begin{eqnarray*}
\max\{\tr XP_{\cX}, X\in \S_+(\cH)\}, \tr_jX=\rho_i\in\rS_{+,1}(\cH), \{i,j\}=\{1,2\}, i=1,2\}.
\end{eqnarray*}
Then Problem \ref{margproddm} is solvable if an only if the above maximum is $1$.
It is possible to convert this problem to an equivalent SDP problem where the admissible set is bounded and has an interior in $S_+(\cH)$, see Section \ref{sub:fdgenSDP}.  Thus one can use interior methods to find the maximum within a given precision $\varepsilon>0$ in polynomial time in the given data and $\varepsilon$. 

 Zhou et al. \cite{Ming} gave necessary and sufficient conditions for the solution of Problem \ref{margproddm}.
These conditions are analogous to the conditions for the solution of Problem \ref{margprob} \cite{Hsu17}. They pointed out  that quantum coupling can be used to extend quantum Hoare logic \cite{MSY11} for proving relational properties between quantum programs and further for verifying quantum cryptographic protocols and differential privacy in quantum computation \cite{ZLMS17}.  The second named author generalized some of the results of \cite{Ming} in \cite{JG19}.

The aim of this paper is to answer Problem \ref{margproddm} to the case when at least one of the Hilbert spaces is an infinite dimensional separable Hilbert space.
The most challenging and interesting parts of this paper are tackling the weak operator convergence in the trace class operators on the tensor product of two Hilbert spaces, (bipartite space), under the partial trace mapping.  As shown in Example \ref{example} the weak operator convergence is not preserved under the partial trace.  This paper
offers some tools and approaches for the quantum marginals problem.  We  hope that our results will be useful to other problems on trace class operators with partial traces.

Our main idea to solve Problem \ref{margproddm} is by stating a countable number of necessary conditions on finite dimensional Hilbert spaces.   Then to show that these conditions are sufficient using compactness arguments.   This was a successful approach in finding infinite dimensional generalizations of Choi's theorem for characterization of quantum channels \cite{Fri18}.

It turns out that the most difficult case is when  $\cH_1,\cH_2,\cX$ are infinite dimensional separable spaces.  We now outline briefly our main result in this case.

Let $\cH$ be an  infinite dimensional separable Hilbert space.  Denote by $\rB(\cH)\supset \rK(\cH)$  the space of bounded linear operators, with the operator norm $\|\cdot\|$,  and the ideal of compact operators respectively.  Let $\rS(\cH)\supset \rS_+(\cH)$ be the subspace of selfadjoint operators and the cone of positive semidefinite operators in $\rB(\cH)$.
Assume that $A\in\rK(\cH)$.  Recall that $A$ has the Schmidt decomposition, which is the singular value decomposition for the the finite dimensional $\cH$, with a nonnegative nonincreasing sequence of singular values $\|A\|=\sigma_1(A)\ge \cdots\ge\sigma_i(A)\ge \cdots \ge0$, which converges to $0$.  For $A\in \rS_+(\cH)\cap \rK(\cH)$ the Schmidt decomposition is the spectral decomposition.  For $p\in[1,\infty)$, denote by
 $\rT_p(\cH)\subset \rK(\cH) $ the Banach space of all compact operators with the $p$-Schatten norm $\|A\|_p=(\sum_{i=1}^{\infty} \sigma_i(A)^p)^{1/p}$.  The Banach space $\rT_1(\cH)$ is the space of trace class operators, which will be abbreviated to $\rT(\cH)$.  For $A\in \rT(\cH)$ the trace $\tr A$ is a bounded linear functional $A\mapsto\tr A$ satisfying $|\tr A|\le \|A\|_1$.  For $A\in \rT(\cH)\cap \rS(\cH)$, $\tr A$ is the sum of the eigenvalues of $A$.
 The cone of positive semidefinite operators in trace class is denoted as $\rT_{+}(\cH)=\rT(\cH)\cap \rS_+(\cH)$.  Note that $\|A\|_1=\tr A$ if and only if $A\in\rT_+(\cH)$.  (See Appendix A.)  Denote by $\rS_{+,1}(\cH)\subset \rT_{+}(\cH)$  the convex set of positive semidefinite trace class operators with trace $1$, i.e., the density operators.

Assume that $\cH=\cH_1\otimes \cH_2$, where $\cH_1$ and $\cH_2$ are separable Hilbert spaces.  Suppose that $\rho\in \rT(\cH)$.  Then there are two partial trace maps: $\tr_1: \rT(\cH)\to \rT(\cH_2)$ and $\tr_2 \rT(\cH)\to\rT(\cH_1)$ which are both contractions: $\|\tr_i(A)\|_1\le \|A\|_1$ for $i=1,2$, see Section \ref{sec:prilHs}.
Denote
\begin{equation}\label{defPhiid}
\Phi: \rT(\cH)\to \rT(\cH_1)\oplus\rT(\cH_2), \quad \Phi(\rho)=(\tr_2\rho,\tr_1\rho).
\end{equation}
Then $\|\Phi\|\le 2$.  Let $\Sigma=\Phi(\rT_+(\cH))$.  Then
\begin{equation}\label{defSigma}
\Sigma=\{(\rho_1,\rho_2) ~|~ \rho_1\in \rT_+(\cH_1), \rho_2 \in \rT_+(\cH_2), \,\tr\rho_1=\tr\rho_2\}.
\end{equation}
For $(\rho_1,\rho_2)\in\Sigma$, let
\begin{equation}\label{defcM}
\cM(\rho_1,\rho_2)=\Phi^{-1}(\rho_1,\rho_2)\cap\rT_+(\cH).
\end{equation}
Thus if $\rho_1,\rho_2$ are density operators then $\cM(\rho_1,\rho_2)$ is the convex set of bipartite density operates with marginals $\rho_1,\rho_2$.  Observe that $\rT_+(\cH)$ fibers over $\Sigma$: $\rT_+(\cH)=\cup_{(\rho_1,\rho_2)\in\Sigma}\cM(\rho_1,\rho_2)$.

 We now state the main result of this paper:
 \begin{theorem}\label{XinfH2fin}  Suppose that $\cH_1$ and $\cH_2$ are infinite dimensional separable Hilbert spaces.
Assume that $\rho_i\in \rS_{+,1}(\cH_i)$ for $i=1,2$.  Let $\cH=\cH_1\otimes\cH_2$.  Define on $\rT(\cH)$ the following Lipschitz convex function with respect to $\|\cdot \|_1$:
\begin{eqnarray}\label{deffuncf}
f(X)=\|\tr_2 X-\rho_1\|_1+\|\tr_1 X -\rho_2 \|_1.
\end{eqnarray}
Suppose that $\cX\subset\cH$ is infinite dimensional closed subspace, with an orthonormal basis $\bx_i, i\in\N$.  Let $\cX_n$ be the subspace spanned by $\bx_1,\ldots,\bx_n$ for $n\in\N$.
Consider the minimization problem
\begin{eqnarray}\label{defminSX1n}
\mu_n(\rho_1,\rho_2)=\inf\{f(X), \; X\in \rS_{+}(\cX_{n})\},
\end{eqnarray}
for $n\in \N$.
This infimum is attained for some $X_n\in \cX_n$ which satisfies $\|X_n\|\le 2$.
Then there exists $\rho\in \rS_{+,1}(\cH), \supp \rho\subseteq \cX$ such that  $\tr_2 \rho=\rho_1, \tr_1\rho=\rho_2$ if and only if
\begin{eqnarray}\label{munrhocond}
\lim_{n\to\infty} \mu_n(\rho_1,\rho_2)=0.
\end{eqnarray}
\end{theorem}

We now comment on the above theorem.  The Lipschitz and convexity properties of $f$ on $\rT(\cH)$ follows straightforward from the triangle inequality for norms and the fact that the partial traces are contractions.  Since $\cX_n$ has dimension $n$ the minimum $\mu_n(\rho_1,\rho_2)$ can be computed efficiently.  Furthermore, the sequence $\mu_n(\rho_1,\rho_2)$ is decreasing.  It is also straightforward to show that that if there exists $\rho\in\rT_1(\cH), \supp \rho\subseteq \cX$ such that $\tr_1\rho=\rho_2$ and $\tr_2\rho=\rho_1$ then \eqref{munrhocond} holds.  The nontrivial part of the above theorem is that the condition \eqref{munrhocond} yields the existence of $\rho$.
This part follows from the following nontrivial interesting result:
\begin{theorem}\label{wotnrmconv}
Assume that $\cH_1$ and $\cH_2$ are infinite dimensional separable Hilbert spaces. Suppose that $\rho_i\in\rT_+(\cH_i)$ for $i=1,2$.
Assume that the sequence  $\rho^{(n)}\in\rT_{+}(\cH), n\in\N$ converges in the weak operator topology to $\rho\in \rT_{+}(\cH)$.  Suppose furthermore that
\begin{eqnarray}\label{partrrhoncond}
\lim_{n\to\infty} \|\tr_1 \rho^{(n)}-\rho_2\|_1+\|\tr_2 \rho^{(n)}-\rho_1\|_1=0.
\end{eqnarray}
Then
\begin{eqnarray}\label{rhonnrmconvrho}
\lim_{n\to\infty} \|\rho^{(n)}-\rho\|_1=0.
\end{eqnarray}
In particular $\tr\rho=\tr\rho_1=\tr\rho_2$.  Hence $\rho$ is a density operator if and only if $\rho_1$ and $\rho_2$ are density operators.
\end{theorem}
Our proof is long and computational.  Perhaps there exists a short simple proof of this theorem.

The above theorem implies the following results.
First, $\cM(\rho_1,\rho_2)$ is a compact metric space with the distance  induced by the norm $\|\cdot\|_1$ on $\rT_+(\cH_1\otimes\cH_2)$.
Second, Theorem \ref{hmetricfib}  shows that this Hausdorff distance $\textrm{hd}(\cM(\rho_1,\rho_2),\cM(\sigma_1,\sigma_2))$ is a complete metric on the fibers $\cM(\rho_1,\rho_2)$.

In this paper we use many standard and known results for compact operators, trace class operators and Hilbert-Schmidt operators, ($\rT_2(\cH)$), on a separable Hilbert space.  For convenience of the reader, we tried to make this paper self contained as much as possible.  We stated some of the known and less known results that we used in two Appendices.

We now survey briefly the content of this paper.   In Section \ref{sec:prilHs} we discuss
some basic results on operators on separable Hilbert spaces.  We recall the Schmidt decomposition of compact operators and its properties.  We discuss in detail the trace class operators $\rT(\cH)$, the Hilbert-Schmidt operators $\rT_2(\cH)$ and relations between these Banach spaces.  Next we consider these classes of operators for bipartite Hilbert space $\cH=\cH_1\otimes \cH_2$.  We discuss in detail the partial trace operators and their properties under the weak operator convergence.

In Section \ref{sec:proofmt} we give proofs to Theorems \ref{wotnrmconv} and \ref{XinfH2fin}.   Most of this Section is devoted to the proof of Theorem \ref{wotnrmconv}, which is long and computational.  The proof of Theorem \ref{XinfH2fin} follows quite simply from Theorem \ref{wotnrmconv}.

Section {\ref{sec:dfcX} discusses a simpler case of quantum marginals problem, where the support of $\rho$ is contained in a finite dimensional subspace $\cX$ of the bipartite space $\cH$.  In this case we can replace the minimum problem \eqref{defminSX1n}, which boils down to the minimum of Lipschitz convex function on a finite dimensional compact convex set,  to a maximum problem in semidefinite programming(SDP), on a bounded compact set of positive semidefinite matrices, which has an interior.  In Subsection \ref{sub:fdgenSDP} we discuss a more general
SDP problem than the one considered in \cite{Ming}, and its dual problem. Most of Subsection \ref{subsec:infdcase} is devoted to the case where $\cH_1$ and $\cH_2$ are separable infinite dimensional.  The main result of this subsection is Theorem \ref{H1H2inf}, which is an analog of Theorem \ref{XinfH2fin}, where $\mu_n(\rho_1,\rho_2)$ is replaced by $\mu_n(\rho_1,\rho_2,\cX)$, which is  the maximum of an appropriate SDP problem.

In Section \ref{subsec: Hausdorff}, we prove that the  space of fibers $\cM(\rho_1,\rho_2)$ over $\Sigma$ is a compact metric space with respect to {the} Hausdorff metric.

Appendix \ref{appedixA} is devoted to various inequalities on singular values of compact operators that we use in this paper.  All the results in this Appendix are well known.  Appendix \ref{app:conv} is devoted mostly to the connection of weak operator convergence to weak star convergence on trace class and to the weak convergence on Hilbert-Schmidt operators.  All the results in this Appendix, except perhaps part (2) of Lemma \ref{convT2}, are well known to the experts.
\section{Preliminary results on operators in Hilbert spaces}\label{sec:prilHs}
We now recall some results needed in this paper on operators in a separable Hilbert space $\cH$.  
Our main references are  \cite{PBEB15,Con91,CAM67,NC00,RS98,Sem}.
 For completeness, we outline a short proof of some known results which do not appear in \cite{RS98} or \cite{Con91}.
We will follow closely the notions in \cite{Fri18}.  The elements of $\cH$ are denoted by lower bold letters as $\bx$.
We denote the inner product in $\cH$ by $\langle \bx,\by\rangle$, which is linear in $\bx$ and antilinear in $\by$.  The norm $\|\bx\|$ is equal to $\sqrt{\langle \bx,\bx\rangle}$.
We denote by $\cH^\vee$ the dual space of the linear functional on $\cH$.  Recall that
a linear functional $\bbf\in\cH^\vee$ represented by $\by\in \cH$: $\bbf(\bx)=\langle \bx,\by\rangle$ for all $\bx\in \cH$.  We denote this $\bbf$ by $\by^{\vee}$.  Note
\[(a_1\by_1+a_2\by_2)^\vee=\bar a_1\by_1^\vee +\bar a_2\by_2^\vee.\]
Denote by $\N$ the set of positive integers.  For $n\in\N$ we denote $[n]=\{1,\ldots,n\}$, and let $[\infty]=\N$.  Recall that $\cH$ is separable if it has an orthormal basis $\be_i$ for $i\in[N]$, where $N\in\N\cup\{\infty\}$.  In this paper we assume that $\cH$ is separable.  Then $\dim\cH=N$.  Thus $\cH$ is finite dimensional if $N\in\N$.

We denote by $\rB(\cH)$ the space of bounded linear operators $L:\cH\to \cH$.  The bounded linear operators are denoted by the capital letters. The operator norm of $L$ is given by $\|L\|=\sup\{\|L\bx\|, |\bx|\le 1\}$.  The adjoint operator of $L$ is denoted by $L^\vee$ and is given by the equality $\langle L\bx,\by\rangle=\langle \bx,L^\vee\by\rangle$. $L$ is called a selfadjoint operator if $L^\vee=L$.  Denote  by $\rS(\cH)\subset \rB(\cH)$ the real space of selfadjoint operators.  $L\in \rS(\cH)$ is called nonnegative (positive) if $\langle L\bx,\bx\rangle\ge 0$ ($\langle L\bx,\bx\rangle > 0$) for all $\bx\ne \0$.  Denote  by $\rS_{++}(\cH)\subset \rS_+(\cH)$ the open set of positive
and nonnegative (selfadjoint bounded) operators.  So $\rS_+(\cH)$ is a closed cone and $\rS_{++}(\cH)$ its interior.  Recall that $L\in \rS_+(\cH)$ has a unique root $L^{1/2}\in \rS_+(\cH)$.  If $L$ is positive then $L^{1/2}$ is positive.
For $L\in\rB(\cH)$ we have that $L^\vee L, LL^\vee\in \rS_+(\cH)$, and $|L|=(L^\vee L)^{1/2}\in\rS_+(\cH)$.
 For $A,B\in \rS(\cH)$ we denote $A\succeq B (A\succ B)$ if $A-B\in \rS_+(\cH) (A-B\in \rS_{++}(\cH))$.

$L$ is called rank one operator if $L=\bx\by^\vee$, where $\bx,\by\ne \0$.  Thus $L(\bz)=\langle \bz,\by\rangle \bx$.  $L$ is selfadjoint if and only if $\by=a\bx$ for some $a\in\R$.
$L\in \rS_+(\cH)$ if and only if $a\ge 0$.

Assume that $\dim\cH=\infty$.
Denote  by $\rK(\cH)$ the closed ideal (left and right) of compact operators.  The operator $L$ is in $\rK(\cH)$ if and only if $L$ has singular value decompostion (SVD), (or Schmidt decomposition):
\begin{eqnarray}\label{SVD}
&&L=\sum_{i=1}^{\infty} \sigma_i(L)\bg_i\bbf_i^\vee, \\ &&\|L\|=\sigma_1(L)\ge \cdots\ge\sigma_n(L)\ge\cdots\ge 0, \lim_{i\to\infty} \sigma_i(L)=0.\notag
\end{eqnarray}
Here $\{\bg_1,\ldots, \bg_n, \ldots\}, \{\bbf_1,\ldots, \bbf_n, \ldots\}$ are two orthonormal sets of vectors of $\cH$.  The $n$th singular value of $L$ denoted by  $\sigma_n(L)$, and $\bg_n,\bbf_n$ are called  left and right $n$th singular vectors of $L$.
$L$ is selfadjoint if and only if $\bbf_i=\varepsilon_i \bg_i, \varepsilon_i\in\{-1,1\}$ for all $i\in\N$.  Then \eqref{SVD} is the spectral decomposition of $L$ where $\varepsilon_i\sigma_i(L)$ is the eigenvalue of $L$ with the corresponding eigenvector $\bg_i$.  Furthermore $L\in \rS_+(\cH)\cap \rK(\cH)$ if and only if $\bbf_i=\bg_i$ for all $i\in\N$.
Hence all positive $\sigma_i(L)^2$ are the positive eigenvalues of  compact operators  $LL^\vee,L^\vee L \in \rS_+(\H)\cap \rK(\cH)$.
Note that
\begin{eqnarray*}
\|L-\sum_{i=1}^{n} \sigma_i(L)\bg_i\bbf_i^\vee\|=\sigma_{n+1}(L), \quad n\in\N.
\end{eqnarray*}
Recall that if $A\in \rB(\cH)$ and $L\in \rK(\cH)$ then $AL, LA\in \rK(\cH)$.  Furthermore, one has the inequalities
\begin{eqnarray}\label{SVDinprod}
\sigma_i(AL), \sigma_i(LA)\le \sigma_i(L)\|A\|, \; i\in\N.
\end{eqnarray}
(See Appendix A or \cite[1.11 Theorem]{Con91}.)
The above inequalities on singular values yield that if $L\in \rT(\cH)$ then $AL, LA\in \rT(\cH)$.  、Furthermore, $\|AL\|_1,\|LA\|_1\le \|L\|_1\|A\|$.

If $L\in \rT(\cH)$, then for each orthonormal basis $\be_i,i\in\N$, we have the inequality $\sum_{i=1}^\infty |\langle L\be_i,\be_i\rangle|\le \|L\|_1$.  (See Lemma \ref{tracein}.)  Furthermore the value of the sum $\sum_{i=1}^\infty \langle L\be_i,\be_i\rangle$ is independent of a choice of the basis, is denoted as the trace of $L$ \cite[1.9 Proposition]{Con91}.
Thus the  SVD decomposition \eqref{SVD} of $L\in \rT(\cH)$ yields that
\begin{eqnarray}\label{traceform}
\tr L=\sum_{i=1}^\infty \sigma_i(L)\langle \bg_i,\bbf_i\rangle.
\end{eqnarray}
Thus $|\tr L|\le \|L\|_1$ and equality holds if and only of  $z L\in \rT_{+}(\cH)$ for some $z\in\C, |z|=1$.  Note that if
$L\in \rS(\cH)\cap \rT(\cH)$ then the trace of $L$ is the sum of the eigenvalues of $L$.   (See Appendix A.)

Next we recall the following known result that we need later:
\begin{eqnarray*}
\tr LA=\tr AL=\tr A^{1/2} L A^{1/2}\ge 0 \textrm{ if }L\in \rT_+(\cH)  \textrm{ and } A\in \rS_+(\cH).
\end{eqnarray*}

Let $x_1,\ldots,x_n\ge 0$.  Then the function $f(p)=(\sum_{i=1}^n x_i^p)^{1/p}$ is a nonincreasing function for $p\in(0,\infty)$.  Hence $\rT_p(\cH)\subset \rT_q(\cH)$ for $1\le p <q<\infty$.  (Usually $\rT_{\infty}(\cH)$ is identified with $\rB(\cH)$.)
In particular,  $\rT(\cH)\subset \rT_2(\cH)$.  The space $\rT_2(\cH)$ is the Hilbert-Schmidt space of compact operators.  Fix an orthonormal basis $\{\be_i\}, i\in\N$.  Then $A_1,A_2\in \rT_2(\cH)$ have representations
\begin{eqnarray*}
A_l=\sum_{i=j=1}^{\infty} a_{ij,l} \be_{i}\be_j^\vee, \quad \|A_l\|_2= (\sum_{i=j=1}^{\infty}|a_{ij,l}|^2)^{1/2},\;l\in[2].
\end{eqnarray*}
Thus $\rT_2(\cH)$ is a Hilbert space with the inner product
\begin{eqnarray*}
\langle A_1,A_2\rangle=\sum_{i=j=1}^\infty a_{ij,1}\bar a_{ij,2}.
\end{eqnarray*}
It is well known that if $A_1,A_2\in \rT_2(\cH)$ then $A_1A_2\in\rT(\cH)$:
\begin{eqnarray*}
A_1 A_2=\sum_{i=j=1}^\infty (\sum_{k=1}^{\infty} a_{ik,1}a_{kj,2})\be_i\be_j^\vee.
\end{eqnarray*}
Furthermore
\begin{eqnarray*}
&&\langle A_1,A_2\rangle=\tr  A_1 A_2^\vee,\quad  \|A_1A_2\|_1\le \|A_1\|_2\| \|A_2\|_2,\\
&&A_1 A_1^\vee\in\rT_+(\cH),\quad \|A_1 A_1^\vee\|_1=\|A_1\|_2^2=\tr A_1A_1^\vee.
\end{eqnarray*}
See Lemma \ref{prodHSop} or \cite[1.8 Proposition]{Con91}.

We next discuss the tensor product $\cH_1\otimes\cH_2$ of two separable Hilbert spaces.  It is called in quantum physics bipartite states.
Assume that the inner product in $\cH_i$ is $\langle \cdot,\cdot \rangle_i$.
Then $\cH_1\otimes\cH_2$ has the induced inner product satisfying the property $\langle \bx\otimes \by, \bu\otimes \bv\rangle=\langle \bx, \bu\rangle_1 \langle \by,\bv\rangle_2$.  We assume that $\cH_l$ has an orthonormal basis $\be_{i,l}, i\in[N_l]$, where $N_l\in \N\cup\{\infty\}$ for $l\in[2]$.  These two orthonormal bases induce the orthonormal basis $\be_{i,1}\otimes\be_{j,2}$ for $i\in[N_1],j\in [N_2]$ in  $\cH_1\otimes\cH_2$ .  A vector $\ba\in  \cH_1\otimes\cH_2$  has the expansion
\begin{eqnarray}\label{baexpan}
\ba=\sum_{i=j=1}^{N_1,N_2} a_{ij}\be_{i,1}\otimes\be_{j,2},\; \|\ba\|=\sqrt{\sum_{i=j=1}^{N_1,N_2}|a_{ij}|^2}<\infty.
\end{eqnarray}
Note that $\ba$ induces two bounded linear operators $A(\ba):\cH_2\to\cH_1$ and $A(\ba)^\vee: \cH_1\to\cH_2$ given by
\begin{eqnarray}\label{defAa}
A(\ba)=\sum_{i=j=1}^{N_1,N_2} a_{ij}\be_{i,1}\be_{j,2}^{\vee}, \quad A(\ba)^\vee=\sum_{i=j=1}^{N_1,N_2} \bar a_{ij}\be_{j,2}\be_{i,1}^{\vee}.
\end{eqnarray}

We can view $A(\ba)$  as a matrix $\hat A=[a_{ij}]_{i=j=1}^{N_1,N_2}$.  We denote by $\hat A^{\dagger}=[a_{pq}^{\dagger}]_{p=q=1}^{N_2,N_1}$, where $a^{\dagger}_{pq}=\bar a_{qp}$ for all $p\in[N_2], q\in[N_1]$.  ($\hat A^{\dagger}$ is the ``transpose conjugate'' of $\hat A$.)  Next we observe that the operators $A(\ba)$ and $A(\ba)^\vee$  can be viewed as adjoint Hilbert-Schmidt operators on $\tilde\cH=\cH_1\oplus \cH_2$, with the inner product:
\begin{eqnarray*}
\langle (\bx,\bu),(\by,\bv)\rangle=\langle \bx,\by\rangle_1+\langle \bu,\bv\rangle_2.
\end{eqnarray*}
Let $\tilde \be_{i,1}=(\be_{i,1},0), \tilde \be_{j,2}=(0,\be_{j,2})$ for $i\in[N_1],j\in[N_2]$.
Define
\begin{eqnarray*}\label{deftAa}
\tilde A(\ba)=\sum_{i=j=1}^{N_1,N_2} a_{ij}\tilde \be_{i,1}\tilde \be_{j,2}^{\vee}, \quad
\tilde A(\ba)^\vee=\sum_{i=j=1}^{N_1,N_2} \bar a_{ij}\tilde \be_{j,2}\tilde \be_{i,1}^{\vee}
\end{eqnarray*}
Then $\tilde A(\ba), \tilde A(\ba)^\vee\in\rT_2(\tilde\cH)$.  Furthermore we have the following relations
\begin{eqnarray*}
\tilde A(\ba)\tilde A(\ba)^\vee\big|\cH_1=A(\ba) A(\ba)^\vee, \quad \tilde A(\ba)\tilde A(\ba)^\vee\big|\cH_2=0,\\
\tilde A(\ba)^\vee\tilde A(\ba)\big|\cH_2=A(\ba)^\vee A(\ba), \quad \tilde A(\ba)^\vee\tilde A(\ba)\big|\cH_1=0.
\end{eqnarray*}
Lemma \ref{prodHSop} yields that $A(\ba)A(\ba)^\vee\in \rT_+(\cH_1), A(\ba)^\vee A(\ba)\in \rT_+(\cH_2)$, and the two operators have the same singular values.  Thus the matrices $\hat A \hat A^\dagger,  \hat A^\dagger \hat A$ represent the operators $A(\ba)A(\ba)^\vee, A(\ba)^\vee A(\ba)$ in the bases $\{\be_{i,1}\}, \{\be_{j,2}\}$ respectively.

Let $\bb=\sum_{i=j=1}^{N_1,N_2} b_{ij}\be_{i,1}\otimes\be_{j,2}\in \cH_1\otimes\cH_2$.  Denote $\hat B=[b_{ij}]_{i=j=1}^{N_1.N_2}$.  Then $\langle \ba,\bb\rangle=\tr \hat A \hat B^{\dagger}=\tr\hat B^{\dagger} \hat A$.

Assume that $F\in \rT(\cH_1\otimes\cH_2)$.  We now discuss the notions of partial traces
$\tr_1(F)\in \rT(\cH_2)$ and  $\tr_2(F)\in \rT(\cH_1)$.  Assume first that $F$ is a rank one product operator: $(\bx\otimes \by) (\bu\otimes\bv)^\vee$.  Then
\begin{eqnarray}\label{tr12rankone1}
&&\tr_1((\bx\otimes \by) (\bu\otimes\bv)^\vee)=\langle \bx,\bu\rangle \by\bv^\vee,\\
&&\tr_2((\bx\otimes \by) (\bu\otimes\bv)^\vee)=\langle \by,\bv\rangle \bx\bu^\vee.
\label{tr12rankone2}
\end{eqnarray}
Hence
\begin{eqnarray*}
&&\|(\bx\otimes \by) (\bu\otimes\bv)^\vee\|_1=\|\bx\|\|\by\|\|\bu\|\|\bv\|, \\
&&\|\tr_1(\bx\otimes \by) (\bu\otimes\bv)^\vee)\|_1=|\langle \bx,\bu\rangle|\|\by\|\|\bv\|,\\
&&\|\tr_2((\bx\otimes \by) (\bu\otimes\bv)^\vee)\|_1=|\langle \by,\bv\rangle| \|\bx\|\|\bu\|.
\end{eqnarray*}
\begin{lemma}\label{tr12ab}  Assume that $\cH_i$ is a separable Hilbert space of dimension $N_i$ with a basis $\be_{j,i}, j\in[N_i]$ for $i\in[2]$.  Denote $\cH=\cH_1\otimes \cH_2$.  Let $\ba,\bb\in\cH$.  Suppose that $\ba$ has the representation \eqref{baexpan}.  Assume that $\bb$ has a similar expansion and $\hat A=[a_{ij}], \hat B=[b_{ij}],i\in [N_1], j\in[N_2] $ are the representation matrices of $\ba,\bb$ respectively.
Denote by $C$ and $D$ the following operators:
\begin{eqnarray}\quad\quad
\tr_2\ba\bb^\vee=C=\sum_{i=p=1}^{N_1}c_{ip}\be_{i,1}\be_{p,1}^\vee,  \tr_1\ba\bb^\vee= D=\sum_{j=q=1}^{N_2}d_{jq}\be_{j,2}\be_{q,2}^\vee.
\end{eqnarray}
Then
\begin{eqnarray}\label{defhatCD}
\hat C=\hat A\hat B^{\dagger}=[c_{ip}]_{i=p=1}^{N_1},\quad \hat D=\hat{A}^{\top} \overline{\hat B}=[d_{jq}]_{j=q=1}^{N_2}.
\end{eqnarray}
 Furthermore $C\in \rT(\cH_1), D\in \rT(\cH_2)$ and the following inequalities and equalities hold
\begin{eqnarray}\label{upbdnmtr12}
&&\max(\|\tr_2 \ba\bb^\vee)\|_1,\|\tr_1 \ba\bb^\vee\|_1)\le \|\ba\| \|\bb\|=\|\ba\bb^\vee\|_1,\\
\label{inprodtr2}
&&\langle (\tr_2 \ba\bb^\vee)\bx,\by\rangle=\sum_{j=1}^{N_2} \langle \bx\otimes\be_{j,2},\bb\rangle\langle \ba,\by\otimes\be_{j,2}\rangle,\; \bx,\by\in \cH_1,\\
&&\langle (\tr_1 \ba\bb^\vee)\bu,\bv\rangle=\sum_{i=1}^{N_1} \langle \be_{i,1}\otimes \bu,\bb\rangle\langle \ba,\be_{i,1}\otimes\bv\rangle,\; \bu,\bv\in \cH_2.
\label{inprodtr1}
\end{eqnarray}
In particular
\begin{eqnarray}\label{eqalltraces}
\tr \ba\bb^\vee=\tr\tr_2  \ba\bb^\vee= \tr\tr_1\ba\bb^\vee=\langle \ba,\bb\rangle.
\end{eqnarray}
\end{lemma}
\begin{proof}  Clearly $\ba\bb^\vee\in \rT(\cH)$.  Furthermore
\begin{eqnarray*}
\|\ba\bb^\vee\|=\|\ba\|\|\bb\|=(\sum_{i=j=1}^{N_1,N_2} |a_{ij}|^2)^{\frac{1}{2}} (\sum_{p=q=1}^{N_1,N_2} |b_{pq}|^2)^{\frac{1}{2}}.
\end{eqnarray*}
Observe next that
\begin{eqnarray*}
&&\ba\bb^\vee=(\sum_{i=j=1}^{N_1,N_2} a_{ij}\be_{i,1}\otimes\be_{j,2})(\sum_{p=q=1}^{N_1,N_2} b_{pq}\be_{p,1}\otimes\be_{q,2})^\vee=\\
&&\sum_{i=j=p=q=1}^{N_1,N_2,N_1,N_2} a_{ij}\bar b_{pq} (\be_{i,1}\otimes\be_{j,2})(\be_{p,1}\otimes \be_{q,2})^\vee.
\end{eqnarray*}
Use \eqref{tr12rankone1} and \eqref{tr12rankone2} to deduce that the operators
 $C=\tr_2(\ba\bb^{\vee})$ and $D=\tr_1(\ba\bb^{\vee})$, which represented by matrices $\hat C$ and $\hat D$ respectively, satisfy \eqref{defhatCD} and \eqref{inprodtr2}-\eqref{inprodtr1}.

Let $\tilde\cH$ and $\tilde A(\ba),\tilde A(\bb)\in\rT_2(\tilde \cH)$ be defined as above.
Then $\hat C$ and $\overline{\hat D}$ represent $\tilde A(\ba) \tilde A(\bb)^\vee\big|\cH_1\in \rT(\cH_1)$ and $A(\ba)^\vee A(\bb)\big|\cH_2\in \rT(\cH_2)$.
This shows that $C$ and $D$ are in the trace class.  Lemma \ref{prodHSop} yields that
\begin{eqnarray*}
&&\|\hat C\|_1=\|\tilde A(\ba) \tilde A(\bb)^\vee\|_1\le \|\tilde A(\ba)\|_2 \|\tilde A(\bb)^\vee\|_2=\|\hat A\|_2\|\hat B\|_2=\|\ba\|\|\bb\|,\\
&&\|\hat D\|_1=\|\tilde A(\ba)^\vee \tilde A(\bb)\|_1\le \|\tilde A(\ba)^\vee\|_2 \|\tilde A(\bb)\|_2=\|\hat A\|_2\|\hat B\|_2=\|\ba\|\|\bb\|.
\end{eqnarray*}
This proves \eqref{upbdnmtr12}.

It is left to show \eqref{eqalltraces}.  As $\sigma_1(\ba\bb^\vee)=\|\ba\|\|\bb\|$ and all other singular values of $\ba\bb^\vee$ are zero \eqref{traceform} yields that $\tr \ba\bb^\vee=\langle \ba,\bb\rangle$.  Observe next
\begin{eqnarray*}
\tr(\tr_2\ba\bb^\vee)=\sum_{i=1}^{N_1}\langle (\tr_2 \ba\bb^\vee)\be_{i,1},\be_{i,1}\rangle=\\
\sum_{i=1}^{N_1}\sum_{j=1}^{N_2} \langle \be_{i,1}\otimes\be_{j,2},\bb\rangle\langle \ba,\be_{i,1}\otimes\be_{j,2}\rangle=\langle \ba,\bb\rangle.
\end{eqnarray*}
The equaity $\tr (\tr_1\ba\bb^\vee)=\langle \ba,\bb\rangle$ follows similarly.
\end{proof}
The following lemma is known, see Theorem 26.7 and its proof in \cite{PBEB15}, and we bring its proof for completeness.
\begin{lemma}\label{tracecont}  Assume that  $F\in \rT(\cH_1\otimes \cH_2)$.  Then
\begin{enumerate}
\item
$\tr_1(F)\in \rT(\cH_2),\tr_2(F)\in \rT(\cH_1)$.
\item   $\|\tr_1(F)\|_1,\|\tr_2(F)\|_1\le \|F\|_1$.
\item $\tr (\tr_1 F)=\tr (\tr_2 F)=\tr F$.
\item Assume that  $F\in \rT_{+}(\cH)$.   Then $\tr_1(F)\in\rT_+(\cH_2), \tr_2(F)\in\rT_+(\cH_1)$ and
\[\|F\|_1=\tr (F)=\tr(\tr_1(F))=\|\tr_1(F)\|_1=\tr(\tr_2(F))=\|\tr_2(F)\|_1.\]
\end{enumerate}
\end{lemma}
\begin{proof}
Assume that $F$ has the following singular value decomposition:
\begin{eqnarray}\label{SVDF}\quad\quad\quad
F=\sum_{k=1}^\infty \sigma_i(F) \ba_k\bb_k^\vee,\langle \ba_k,\ba_l\rangle=\langle \bb_k,\bb_l\rangle, k,l\in\N.
\end{eqnarray}
Then $\|F\|_1=\sum_{k=1}^\infty \sigma_k(F)$.
Hence
\begin{eqnarray*}
&&\tr_2 F=\sum_{k=1}^\infty \sigma_k(F)\tr_2 \ba_k\bb_k^\vee,\\
&&\|\tr_2 F\|_1\le \sum_{k=1}^\infty \sigma_k(F)\|\tr_2 \ba_k\bb_k^\vee\|_1\le \sum_{k=1}^\infty \sigma_k(F)\|\ba_k\|\|\bb_k\|\le\\
&&\sum_{k=1}^\infty \sigma_k(F)=\|F\|_1.
\end{eqnarray*}
This shows that $\tr_2F\in \rT_+(\cH_1)$ and $\|\tr_2 F\|_1\le \|F\|_1$.  Use \eqref{eqalltraces} to deduce that $\tr F=\tr (\tr_2 F)$.  Similar results hold for $\tr_1 F$.

Assume now that $F\succeq 0$.  Then in the decomposition \eqref{SVDF} $\ba_k=\bb_k$ for $k\in\N$.  Use \eqref{eqalltraces} to deduce that
\[\tr F=\sum_{k=1}^\infty \sigma_i(F) \tr \ba_k\ba_k^\vee=\sum_{k=1}^\infty \sigma_i(F) \langle \ba_k,\ba_k\rangle=\sum_{k=1}^\infty \sigma_i(F) =\|F\|_1.\]
We next show that $\tr_2 F\succeq 0$.
Use the equality \eqref{inprodtr2} to deduce
\begin{eqnarray}\label{postr2F}
&&\langle (\tr_2F)\bx,\bx\rangle=\sum_{k=1}^\infty \sum_{j=1}^{N_2} \sigma_k(F)\langle \bx\otimes\be_{j,2},\ba_k\rangle\langle \ba_k,\bx\otimes\be_{j,2}\rangle\ge 0,
\end{eqnarray}
for each $\bx\in\cH_1$.
Hence $\tr_2 F\succeq 0$.  Therefore $\|\tr_2 F\|_1=\tr (\tr_2 F)=\tr F=\|F\|_1$.
Similar arguments apply to $\tr_1 F$.
\end{proof}

\begin{lemma}\label{weakconpartrace} Let $\cH_l$ be a separable Hilbert space of dimension $N_l\in\N\cup\{\infty\}$ for $l\in[2]$.  Set $\cH=\cH_1\otimes\cH_2$.
\begin{enumerate}
\item Assume that $\ba_n, \bb_n,\in\cH, n\in\N$, and  $\ba_n\overset{w}{\to}\ba, \bb_n\overset{w}{\to}\bb$.  Then
\begin{eqnarray}\label{weakconpartrace1}
&&\ba_n\bb_n^\vee\overset{w.o.t.}{\to}\ba\bb^\vee \textrm{ in } \rT(\cH), \\
&&\liminf \tr \ba_n\ba_n^\vee \ge \tr \ba\ba^\vee.
\label{weakconpartrace1in}
\end{eqnarray}
For each $\bx_i\in\cH_i$ for $i\in[2]$  the following inequalities hold
\begin{eqnarray}\label{weakconpartrace2}
&&\liminf \langle (\tr_1 \ba_n\ba_n^\vee) \bx_2,\bx_2\rangle\ge \langle (\tr_1 \ba\ba^\vee) \bx_2,\bx_2\rangle,\\
&&\liminf \langle (\tr_2 \ba_n\ba_n^\vee) \bx_1,\bx_1\rangle\ge \langle (\tr_2 \ba\ba^\vee) \bx_1,\bx_1\rangle.\notag
\end{eqnarray}
Assume that for $l\in[2]$ $N_l$ is finite.
Then
\begin{eqnarray}\label{trleq}
 \tr_l \ba_n\bb_n^\vee \overset{w.o.t.}{\to}\tr_l \ba\bb^\vee.
 \end{eqnarray}
\item Assume that the sequence $\rho_n\in \rT_+(\cH)$ converges in weak operator topology to $\rho\in\rT(\cH)$.  Then $\rho\in T_+(\cH)$ and the following conditions hold:
\begin{eqnarray}\label{totaltraceinrhonrho}
&&\liminf \tr\rho_n\ge \tr\rho,\\
&&\lim_{n\to\infty}\tr\rho_n=\tr\rho \iff \lim_{n\to\infty}\|\rho_n-\rho\|_1=0,
\label{limtraceinrhonrho}\\
\label{p1traceinrhonrho}
&&\liminf \langle (\tr_1 \rho_n) \bx_2,\bx_2\rangle\ge \langle (\tr_1 \rho) \bx_2,\bx_2\rangle, \\
\label{p2traceinrhonrho}
&&\liminf \langle (\tr_2 \rho_n) \bx_1,\bx_1\rangle\ge \langle (\tr_2 \rho) \bx_1,\bx_1\rangle.
 \end{eqnarray}
 If $N_l$ is finite then
 \begin{eqnarray}\label{trleqrho}
 \tr_l \rho_n \overset{w.o.t.}{\to}\tr_l \rho.
 \end{eqnarray}
\end{enumerate}
\end{lemma}
\begin{proof} (1)
For each $\bu,\bv\in\cH$ we have the equality $\langle (\ba_n\bb_n^\vee)\bu,\bv\rangle = \langle \bu,\bb_n\rangle \langle \ba_n,\bv\rangle$,  As $\ba_n\overset{w}{\to}\ba, \bb_n\overset{w}{\to}\bb$ we deduce \eqref{weakconpartrace1}.
 Recall that
$\liminf \|\ba_n\|\ge \|\ba\|$.  As $\tr \bc\bc^\vee=\|\bc\|^2$ for $\bc\in\cH$ we deduce \eqref{weakconpartrace1in}.

Assume that $N_2$ is finite.  We prove \eqref{trleq} for $l=2$.  Recall \eqref{inprodtr2} for $\ba_n,\bb_n$:
\begin{eqnarray*}
\langle (\tr_2 \ba_n\bb_n^\vee)\bx,\by\rangle=\sum_{j=1}^{N_2} \langle \bx\otimes\be_{j,2},\bb_n\rangle\langle \ba_n,\by\otimes\be_{j,2}\rangle,\; \bx,\by\in \cH_1
\end{eqnarray*}
Letting $n\to\infty$ we get \eqref{inprodtr2}.  Hence \eqref{trleq} holds for $l=2$.  Similar arguments apply if $N_1$ is finite.

We now show \eqref{weakconpartrace2}.  Assume first that $N_2$ is finite.  Then \eqref{trleq} yields the equality in \eqref{weakconpartrace2}.   Assume that $N_2=\infty$.  Choose $N\in\N$ and let $L_{n,N}$ and $L_N$ be the following finite rank operators in $\rT(\cH_1)$:
\begin{eqnarray*}
&&\langle L_{n,N}\bx,\by\rangle=\sum_{j=1}^{N} \langle \bx\otimes\be_{j,2},\ba_n\rangle\langle \ba_n,\by\otimes\be_{j,2}\rangle,\\
&&\langle L_{N}\bx,\by\rangle=\sum_{j=1}^{N} \langle \bx\otimes\be_{j,2},\ba\rangle\langle \ba,\by\otimes\be_{j,2}\rangle.
\end{eqnarray*}
Clearly, the sequence $L_{n,N}, n\in\N$ converges in weak operator topology to $L_N$ for each $N\in\N$.
Observe next
\[\langle(\tr_2\ba_n\ba_n^\vee)\bx,\bx\rangle=\sum_{j=1}^{\infty} |\langle \ba_n,\bx\otimes\be_{j,2} \rangle|^2\ge \langle L_{n,N}\bx,\bx\rangle.\]
Hence
\[\liminf \langle (\tr_2 \ba_n\ba_n^\vee) \bx,\bx\rangle\ge \langle L_N \bx,\bx\rangle.\]
As $\lim_{N\to\infty}  \langle L_N \bx,\bx\rangle= \langle (\tr_2\ba\ba^\vee) \bx,\bx\rangle$ we deduce the second inequality in\eqref{weakconpartrace2}.
Similarly we deduce the first inequality in\eqref{weakconpartrace2}.

\noindent (2)  The claim that $\rho\in\rT_+(\cH)$ and the inequality  \eqref{totaltraceinrhonrho} follow from Lemma \ref{weaktoplemma}.  To show other claims in part (2) of the lemma we repeat some arguments of the proof of Lemma \ref{weaktoplemma}.
 Assume that the spectral decomposition of $\rho_n$ is $\sum_{k=1}^\infty \sigma_k(\rho_n)\ba_{k,n}\ba_{k,n}^\vee$.   Fix $\bx_l\in\cH_l$ for $l\in[2]$.  We first choose a subsequence $n_p,p\in \N$ such that a particular $\liminf$ stated in part (2) of the lemma is achieved for this subsequence.  Clearly $\rho_{n_p}\overset{w.o.t.}{\to}\rho$.  Hence, without loss of generality we can assume that $n_p=p$ for $p\in\N$.
We choose a subsequence $n_m,m\in\N$ such that
\[\lim_{m\to\infty} \sigma_k(\rho_{n_m})=\sigma_k, \quad \ba_{k,n_m}\overset{w}{\to}\ba_k, \quad k\in\N.\]
As $\rho_{n_m}$ converges weakly also to $\rho$ we deduce that
$\rho=\sum_{k=1}^\infty\sigma_n\ba_k\ba_k^\vee$ and $\|\ba_k\|\le 1$  for $k\in\N$.
Fix $\varepsilon >0$.  Then there exists $N=N(\varepsilon)$ such that $\sum_{k=N}^\infty \sigma_k<\varepsilon$.  Furthermore there exists $k>K_2(\varepsilon)$ such that $\sigma_N(\rho_k)<\varepsilon$.  Let
\begin{eqnarray*}
&&B_n=\sum_{k=1}^N \sigma_k(\rho_n)\ba_{k,n}\ba_{k,n}^\vee, \quad C_n=\sum_{k=N+1}^\infty \sigma_k(\rho_n)\ba_{k,n}\ba_{k,n}^\vee,\\
&&B=\sum_{k=1}^N \sigma_k\ba_{k}\ba_{k}^\vee, \quad C=\sum_{k=N+1}^\infty \sigma_k\ba_{k}\ba_{k}^\vee
\end{eqnarray*}
Then
\begin{eqnarray*}
&&\rho_n=B_n+C_n,\; \rho=B+C,\;  B_n,C_n,B,C\in T_+(\cH),\; \|\rho-B\|_1=\|C\|_1<\varepsilon,\\
&& \tr_l\rho_n\succeq \tr_lB_n, \; \|\tr_l \rho -\tr_l B\|_1=\|\tr _l C\|_1\le \|C\|_1<\varepsilon,\; n\in\N,l\in[2].
\end{eqnarray*}
For $l\in[2]$ let $\{l'\}=[2]\setminus\{l\}$.
Part (1) yields that
\begin{eqnarray*}
&&\liminf \langle (\tr_{l'} \rho_n) \bx_{l},\bx_l\rangle\ge\liminf \langle (\tr_{l'} B_n) \bx_{l},\bx_l\rangle\ge\\
&&\langle (\tr_{l'} B) \bx_l,\bx_l\rangle\ge  \langle (\tr_{l'} \rho) \bx_l,\bx_l\rangle -\varepsilon\|\bx_l\|^2.
\end{eqnarray*}
As $\varepsilon>0$ can be chosen  arbitrary small we deduce all the inequalities in part (2).   The condition \eqref{limtraceinrhonrho} is \cite[Lemma 4.3]{Dav69}, or  Lemma \ref{weaknormconv1}.

Assume that $N_2$ is finite.  Then $\cH$ is isometric to the direct sum of  $N_2$ copies of $\cH_{1}$. Where each copy $\cH_{1,j}$ has the basis $\be_{i,1}\otimes\be_{j,2}$ for $i\in[N_1]$.  Let $\rho_{n,j};\cH_{1,j}\to \cH_{1,j}$  be the restriction of the sesquilinear form $\langle\rho_n \bu,\bv\rangle$, where $\bu=\bx\otimes\be_{j,2}, \bv=\by\otimes\be_{j,2}$.  Observe that $\tr_2\rho_n=\sum_{j=1}^{N_2}\rho_{n,j}$.
Define similarly $\rho^{(j)}$ for $j\in[N_2]$.  Clearly, $\rho_{n,j}\overset{w.o.t.}{\to} \rho^{(j)}$ for $j\in [N_2]$.  Hence $\tr_2 \rho_n\overset{w.o.t.}{\to}\tr_2\rho=\sum_{j=1}^{N_2} \rho^{(j)}$.  Similar results apply if $N_1$ is finite.
\end{proof}

We now give a simple example to show that in part (1) of Lemma \ref{weakconpartrace} we may have strict inequalities.
\begin{example}\label{example}
Assume that $N_1=\infty$.  Consider $\rho_n=(\be_n\otimes\be_1)(\be_n\otimes\be_1)^\vee, n\in\N$.  Then $\be_n\otimes\be_1\overset{w.o.t.}{\to}{\0}$.  So $\rho_n \overset{w.o.t.}{\to} \rho=0$.  Clearly $\tr_2(\rho_n)=\be_n\be_n^\vee\overset{w.o.t.}{\to} 0$, and $\tr_1 \rho_n=\be_1\be_1^\vee$.  Thus $\tr_1 \rho_n$ does not converge weakly to $\tr_1\rho$.
\end{example}
\section{Proof of the main theorems}\label{sec:proofmt}
\subsection{Proof of Theorem \ref{wotnrmconv}}\label{subsec:prfptconv}
As
\begin{eqnarray*}
\|\rho^{(n)}\|_1=\tr\rho^{(n)}=\tr(\tr_1\rho^{(n)})=\tr(\tr_2\rho^{(n)}), \quad \|\rho_i\|_1=\tr\rho_i, i\in[2],
\end{eqnarray*}
we deduce that $\tr\rho_1=\tr\rho_2=\lim_{n\to\infty}\tr\rho^{(n)}$.   Lemma \ref{weaktoplemma} yields that $\tr\rho_1 \ge \tr \rho$.  Lemma \ref{weaknormconv1}
implies that $\lim_{n\to\infty}\|\rho^{(n)}-\rho\|_1=0$ if and only if $\tr\rho_1=\tr\rho$.  Assume to the contrary that $\tr\rho_1=\tr\rho_2>\tr\rho$.

The next claims follow from the results in Appendix \ref{app:conv}.  Recall that $\rT(\cH)\subset \rT_2(\cH)$.  Thus $\rho^{(n)},n\in\N$ and $\rho$ are in $\rT_2(\cH)$.  Hence $\rho^{(n)},n\in\N$ converges in the weak topology to $\rho$ in the Hilbert space $\rT_2(\cH)$.  Banach-Sacks theorem \cite{BanachSaks} yields that there exists a subsequence $n_j, j\in\N$ such that the Ces\`{a}ro subsequence
$\hat \rho_m=\frac{1}{m}\sum_{j=1}^m \rho^{(n_j)}, m\in\N$ converges in the norm $\|\cdot\|_2$ to $\rho$.  It is straightforward to show that
\begin{eqnarray*}
\lim_{m\to\infty}\|\tr_2\hat\rho_m-\rho_1\|_1+\|\tr_1\hat\rho_m-\rho_2\|_1=0.
\end{eqnarray*}

The inequalities \eqref{p2traceinrhonrho} and \eqref{p1traceinrhonrho} yield that
\begin{eqnarray*}
\alpha_1=\rho_1-\tr_2\rho \in \rT_+(\cH_1), \quad \alpha_2=\rho_2-\tr_1\rho \in \rT_+(\cH_2).
\end{eqnarray*}
Note that $\tr\alpha_1=\tr\alpha_2>0$.  Consider the spectral decompositions of $\alpha_1$ and $\alpha_2$:
\begin{eqnarray*}
&&\alpha_1=\sum_{i=1}^\infty \sigma_{i,1}\bg_i\bg_i^\vee, \{\sigma_{i,1}\geq 0 \}\searrow 0, \langle \bg_i,\bg_j\rangle=\delta_{ij}, i,j\in\N, \tr\alpha_1=\sum_{i=1}^\infty \sigma_{i,1},\\
&&\alpha_2=\sum_{i=1}^\infty \sigma_{i,2}\bbf_i\bbf_i^\vee, \{\sigma_{i,2}\geq 0\}\searrow 0, \langle \bbf_i,\bbf_j\rangle=\delta_{ij}, i,j\in\N, \tr\alpha_2=\sum_{i=1}^\infty \sigma_{i,2}.
\end{eqnarray*}

 As $\tr\alpha_1=\tr\alpha_2>0$ there  exists $\delta>0$, such that
\begin{equation}\label{deltavalue}
\sigma_{1,1}>\delta, \, \,
\sigma_{1,2}>\delta, \quad \delta>0.
\end{equation}
For $N\in\N$ let
\begin{eqnarray*}
&&\alpha_{N,1}=\sum_{i=1}^{N}\sigma_{i,1} \bg_i \bg_i^{\vee}, \,  \, \tilde\alpha_{N,1}=\sum_{i=N+1}^{\infty}\sigma_{i,1} \bg_i \bg_i^{\vee},\\ &&\alpha_{N,2}=\sum_{i=1}^{N}\sigma_{i,2} \bbf_i \bbf_i^{\vee}, \, \, \tilde\alpha_{N,2}=\sum_{i=N+1}^{\infty}\sigma_{i,2} \bbf_i \bbf_i^{\vee}.
\end{eqnarray*}
Fix $N$ big enough so that
\begin{eqnarray}\label{Nchoice}
\max(\|\tilde\alpha_{N,1}\|_1, \|\tilde\alpha_{N,2}\|_1)<\delta/10.
\end{eqnarray}

For simplicity of the exposition of the proof we consider the following most difficult case.  First, $\alpha_1$ and $\alpha_2$ are not finite dimensional: $\sigma_{i,1}, \sigma_{i,2}>0$ for all $i\in N$.
Second, let $\tilde\cH_1$ and $\tilde\cH_2$ be the closure of subspaces spanned by $\bg_i, i\in\N$ and $\bbf_i,i\in\N$ respectively.
Let $\hat\cH_i$ be the orthogonal complement of $\tilde \cH_i$ in $\cH_i$ for $i\in [2]$.  Then $\hat \cH_1$ and $\hat\cH_2$ are infinite dimensional with orthonormal bases $\hat \bg_i, i\in\N$ and $\hat\bbf_i,i\in\N$ respectively.
Then $\be_{i,j}, i\in\N$ is an orthonormal basis for $\cH_j$ for $j\in [2]$, where
\begin{eqnarray}\label{orthnbH12}\quad\quad
\be_{2i-1,1}=\bg_i,\; \be_{2i,1}=\hat \bg_i, \quad \be_{2i-1,2}=\bbf_i,\;\be_{2i,2}=\hat\bbf_i, \textrm{ for }i\in\N,j\in[2].
\end{eqnarray}
For $m\in\N$, let $P_{m,j}$  be the orthogonal projection in $\cH_j$ on the subspace spanned by $\be_{i,j}, i\in[2m]$ for $j\in[2]$.  Define  $R_m=P_{m,1}\otimes P_{m,2}$ for $m\in\N$.  Then $P_{m,1}, P_{m,2}, R_m$ converge to the identity operators in the strong operator topology in $\cH_1,\cH_2,\cH$ respectively.  Recall  \cite[Lemma 5]{Fri18}:
\begin{eqnarray*}
\lim_{m\to\infty} \|P_{m,1}\beta_1P_{m,1}-\beta_1\|_1+\|P_{m,2}\beta_2P_{m,2}-\beta_2\|_1+\|R_m\beta R_m-\beta\|_1=0
\end{eqnarray*}
for all $\beta_i\in \rT(\cH_i), i\in[2]$ and  $\beta\in\rT(\cH)$.

 Assume that we have the spectral decompositions
\begin{eqnarray}\label{harrhonexp}\quad\quad
\hat\rho_n=\sum_{i=1}^\infty \lambda_{i,n} \bx_{i,n}\bx_{i,n}^\vee, \{\lambda_{i,n}\}\searrow 0, \langle \bx_{i,n}, \bx_{j,n}\rangle=\delta_{ij},  \tr \hat\rho_n=\sum_{i=1}^\infty \lambda_{i,n},\\\notag
\rho=\sum_{i=1}^\infty \lambda_{i} \bx_{i}\bx_{i}^\vee, \{\lambda_{i}\}\searrow 0, \langle \bx_{i}, \bx_{j}\rangle=\delta_{ij},  \tr \rho=\sum_{i=1}^\infty \lambda_{i}.
\end{eqnarray}
Lemma \ref{convT2} yields that $\lim_{n\to\infty}\lambda_{i,n}=\lambda_i$ for each $i\in\N$.
Furthermore, by passing to a subsequence of $\hat\rho_n$, we can assume that $\lim_{n\to\infty} \|\bx_{i,n}-\bx_{i}\|=0$ for each $\lambda_i>0$.
Again, for simplicity of the exposition of the proof we will assume the most difficult case that $\lambda_i>0$ for each $i\in\N$.

Recall that $\lim_{m\to\infty}\|R_m\rho  R_m-\rho\|_1=0$.
Then there exists $m\in\N$ such that
\begin{eqnarray}\label{Rmrhoin}
\|R_m\rho  R_m-\rho\|_1<\delta/10 \textrm{ and } m>N.
\end{eqnarray}

We now keep $m>N$ fixed.  The inequality  \eqref{SVDinprod} yields
\begin{eqnarray*}
&&\sigma_i(R_m(\hat\rho_n-\rho)R_m)\le \|R_m\|^2\sigma_i(\hat\rho_n-\rho)= \sigma_i(\hat\rho_n-\rho), \textrm{ for }i\in\N\Rightarrow \\
&&\|R_m(\hat\rho_n-\rho)R_m)\|_2\le \|\hat\rho_n-\rho\|_2.
\end{eqnarray*}
As $\lim_{n\to\infty}\|\hat\rho_n-\rho\|_2=0$ we deduce that there exists $M_1\in\N$ such that $\|R_m(\hat\rho_n-\rho)R_m)\|_2\le \delta/(20 m)$ for $n>M_1$.  Recall that $\rank R_m=4m^2$.
Hence $\rank R_m(\hat\rho_n-\rho)R_m\le 4m^2$.  Thus
\begin{eqnarray*}
&&\|R_m(\hat\rho_n-\rho)R_m)\|_1= \sum_{i=1}^{4m^2}\sigma_i(R_m(\hat\rho_n-\rho)R_m)\le \\
&&2m \left(\sum_{i=1}^{4m^2}\sigma_i^2(R_m(\hat\rho_n-\rho)R_m)\right)^{1/2}=
2m\|R_m(\hat\rho_n-\rho)R_m)\|_2\le \delta/10.
\end{eqnarray*}

Part (2) of Lemma \ref{tracecont} yields
\begin{eqnarray}\label{trhatrhoarhoin}
\quad\quad \|\tr_i R_m\hat\rho_nR_m-\tr_i R_m\rho R_m\|_1\le\| R_m\hat\rho_nR_m- R_m\rho R_m\|_1\le \delta/10
\end{eqnarray}
 for $n>M_1$.

In addition, we have $\tr_i\hat\rho_n$ converge in trace norm to $\rho_{i+1}$, where $\rho_3=\rho_1$.  Thus there exists  $M_2$, when $n>M_2$, we have
$$\|\tr_i\hat\rho_n-\rho_{i+1}\|_1\le\delta/10, \textrm{ for } i\in[2].$$

Thus for $n>\max(M_1,M_2)$, we have
\begin{eqnarray}\label{inequality1}
\|\tr_i(\hat\rho_n-R_m\hat\rho_n R_m)-(\rho_{i+1}-\tr_i (R_m\rho R_m))\|_1\le \delta/5.
\end{eqnarray}

Lemma \ref{tracecont} and \eqref{Rmrhoin} imply
\begin{eqnarray}\label{mtrinR}
\|\tr_i (R_m\rho R_m)-\tr_i\rho\|_1\le\|R_m\rho  R_m-\rho\|_1<\delta/10 \textrm{ for }i\in[2].
\end{eqnarray}

 We use $\tr_i \rho$ to replace the $\tr_i (R_m\rho  R_m)$ in (\ref{inequality1}) to get
\begin{eqnarray*}
\|\tr_i(\hat\rho_n-R_m\hat\rho_n R_m)-(\rho_{i+1}-\tr_i\rho )\|_1\le3\delta/10.
\end{eqnarray*}
Let $\tr_0$ stand for $\tr_2$.
Recall that   $\alpha_i=\rho_i-\tr_{i-1}\rho=\alpha_{N,i}+\tilde\alpha_{N,i}$ for $i\in[2]$.
The inequality \eqref{Nchoice} yields
\begin{eqnarray}\label{inequalitytr1}
\|\tr_{i-1}(\hat\rho_n-R_m\hat\rho_n R_m)-\alpha_{N,i}\|_1\le 2\delta/5 \textrm{ for }i\in[2]
\end{eqnarray}
and $n>\max(M_1,M_2)$.
We finally get the contradiction by showing that the above two inequalities are incompatible.

Recall the spectral decomposition of $\hat\rho_n$ given by \eqref{harrhonexp}.
Using the bases of $\cH_1,\cH_2$ defined by \eqref{orthnbH12}, we can write
$$\bx_{i,n}=\sum_{p,q=1}^{\infty}\mu_{p,q}^{i,n}\be_{p,1}\otimes \be_{q,2}. $$

So we have
\begin{eqnarray*}
\lambda_{i,n}\bx_{i,n}\bx_{i,n}^\vee&=&\lambda_{i,n}(\sum_{p,q=1}^{\infty}\mu_{p,q}^{i,n}\be_{p,1}\otimes \be_{q,2})(\sum_{r,s=1}^{\infty}\mu_{r,s}^{i,n}\be_{r,1}\otimes \be_{s,2})^\vee\\
&=&\lambda_{i,n}\left(\sum_{p,q,r,s=1}^{\infty}\mu_{p,q}^{i,n}\bar\mu_{r,s}^{i,n}(\be_{p,1}\otimes \be_{q,2})(\be_{r,1}\otimes \be_{s,2})^\vee\right).
\end{eqnarray*}
Hence
$$\hat\rho_n=\sum_{i=1}^{\infty}\lambda_{i,n}\left(\sum_{p,q,r,s=1}^{\infty}
\mu_{p,q}^{i,n}\bar\mu_{r,s}^{i,n}(\be_{p,1}\otimes \be_{q,2})(\be_{r,1}\otimes \be_{s,2})^\vee\right),$$
\begin{eqnarray*}
\hat\rho_n-R_m \hat\rho_n R_m &=& \sum_{i=1}^{\infty}\lambda_{i,n}\left(\sum_{p,q,r,s=1}^{\infty}\mu_{p,q}^{i,n}
\bar\mu_{r,s}^{i,n}(\be_{p,1}\otimes \be_{q,2})(\be_{r,1}\otimes \be_{s,2})^\vee\right)\\
&-&\sum_{i=1}^{\infty}\lambda_{i,n}\left(\sum_{p,q,r,s=1}^{2m}\mu_{p,q}^{i,n}
\bar\mu_{r,s}^{i,n}(\be_{p,1}\otimes \be_{q,2})(\be_{r,1}\otimes \be_{s,2})^\vee\right)
\end{eqnarray*}

%\end{eqnarray}
Then we have
\begin{eqnarray*}
\tr_1(\bx_{i,n}\bx_{i,n}^\vee)&=&\sum_{p,q,s=1}^\infty \mu_{p,q}^{i,n}\bar\mu_{p,s}^{i,n}\be_{q,2}\be_{s,2}^\vee,          \\
\tr_1\hat\rho_n&=&\sum_{i=1}^{\infty}\lambda_{i,n}(\sum_{p,q,s=1} \mu_{p,q}^{i,n}\bar\mu_{p,s}^{i,n}\be_{q,2}\be_{s,2}^\vee)\\
&=&\sum_{q,s=1}^{\infty}\left( (\sum_{i=1}^{\infty}\sum_{p=1}^{\infty}\lambda_{i,n}\mu_{p,q}^{i,n}
\bar\mu_{p,s}^{i,n})\be_{q,2}\be_{s,2}^\vee\right)\\
\tr_1(\hat\rho_n-R_m\hat\rho_n R_m)&=&\sum_{q,s=1}^{\infty}\left( (\sum_{i=1}^{\infty}\sum_{p=1}^{\infty}\lambda_{i,n}\mu_{p,q}^{i,n}
\bar\mu_{p,s}^{i,n})\be_{q,2}\be_{s,2}^\vee\right)\\
&-&\sum_{q,s=1}^{2m}\left( (\sum_{i=1}^{\infty}\sum_{p=1}^{2m}\lambda_{i,n}\mu_{p,q}^{i,n}\bar\mu_{p,s}^{i,n})
\be_{q,2},\be_{s,2}^\vee\right)
\end{eqnarray*}

Write down the diagonal elements of $\tr_1(\hat\rho_n-R_m\hat\rho_n R_m)$:
\begin{eqnarray}\label{inequation}\quad\quad
&&\sum_{q=1}^{2m}(\sum_{i=1}^{\infty}\sum_{p=2m+1}^{\infty}\lambda_{i,n}\mu_{p,q}^{i,n}
\bar\mu_{p,q}^{i,n})\be_{q,2}\be_{q,2}^\vee+\\
&&\sum_{q=2m+1}^{\infty}(\sum_{i=1}^{\infty}\sum_{p=1}^{\infty}
\lambda_{i,n}\mu_{p,q}^{i,n}\bar\mu_{p,q}^{i,n})\be_{q,2}\be_{q,2}^\vee.
\notag
\end{eqnarray}

As $m>N $ are fixed as mentioned above, and  $n>\max(M_1,M_2)$, according to (\ref{inequalitytr1}), we have
\[\|\tr_1(\hat\rho_n-R_m\hat\rho_n R_m)-\alpha_{N,2}\|_1\le 2\delta/5.\]
Observe that the diagonal elements of $\tr_1(\hat\rho_n-R_m\hat\rho_nR_m)-\alpha_{N,2}$ are:
\begin{eqnarray*}
&&\sum_{t=1}^{N}\big((\sum_{i=1}^{\infty}\sum_{p=2m+1}^{\infty}\lambda_{i,n}|\mu_{p,2t-1}^{i,n}|^2)-\sigma_{t,2}\big) \be_{2t-1,2}\be_{2t-1,2}^\vee+\\
&&\sum_{t=1}^{N}\big(\sum_{i=1}^{\infty}\sum_{p=2m+1}^{\infty}\lambda_{i,n}|\mu_{p,2t}^{i,n}|^2\big) \be_{2t,2}\be_{2t,2}^\vee+\\
&&\sum_{q=2N+1}^{2m}\big(\sum_{i=1}^{\infty}\sum_{p=2m+1}^{\infty}\lambda_{i,n}|\mu_{p,q}^{i,n}|^2\big) \be_{q,2}\be_{q,2}^\vee+\\
&&\sum_{q=2m+1}^{\infty}\big(\sum_{i=1}^{\infty}\sum_{p=1}^{\infty}\lambda_{i,n}|\mu_{p,q}^{i,n}|^2\big) \be_{q,2}\be_{q,2}^\vee.
\end{eqnarray*}
by Lemma \ref{tracein} yields that the absolute values of the diagonal elements of $\tr_1(\hat\rho_n-R_m\hat\rho_n R_m)-\alpha_{N,2}$ are bounded by $\|\tr_1(\hat\rho_n-R_m\hat\rho_n R_m)-\alpha_{N,2}\|_1\le 2\delta/5$.
As $\lambda_{i,n}\ge 0$ for $i,n\in\N$ we deduce
\begin{eqnarray*}
&&\sum_{t=1}^{N}|(\sum_{i=1}^{\infty}\sum_{p=2m+1}^{\infty}\lambda_{i,n}|\mu_{p,2t-1}^{i,n}|^2)-\sigma_{t,2}|+
\sum_{t=1}^{N}\sum_{i=1}^{\infty}\sum_{p=2m+1}^{\infty}\lambda_{i,n}|\mu_{p,2t}^{i,n}|^2+\\
&&\sum_{q=2N+1}^{2m}\sum_{i=1}^{\infty}\sum_{p=2m+1}^{\infty}\lambda_{i,n}|\mu_{p,q}^{i,n}|^2+\sum_{q=2m+1}^{\infty}\sum_{i=1}^{\infty}\sum_{p=1}^{\infty}\lambda_{i,n}|\mu_{p,q}^{i,n}|^2\le 2\delta/5.
\end{eqnarray*}
In particular we deduce the following two inequalities:
\begin{eqnarray}\notag
&&|(\sum_{i=1}^{\infty}\sum_{p=2m+1}^{\infty}\lambda_{i,n}|\mu_{p,1}^{i,n}|^2)-\sigma_{1,2}|\le 2\delta/5\\ \label{sec1cin}
&&\sum_{q=2m+1}^{\infty}\sum_{i=1}^{\infty}\lambda_{i,n}|\mu_{1,q}^{i,n}|^2\le 2\delta/5
\end{eqnarray}
The inequality \eqref{deltavalue} and the first above inequality yield
\begin{eqnarray}\label{fir1cin}
\sum_{i=1}^{\infty}\sum_{p=2m+1}^{\infty}\lambda_{i,n}|\mu_{p,1}^{i,n}|^2\ge \delta-2\delta/5=3\delta/5
\end{eqnarray}

Consider now similar inequaities for the diagonal entries of $\tr_2(\hat\rho_n-R_m\hat\rho_nR_m)-\alpha_{N,1}$.  Then the analogous inequality to \eqref{sec1cin} is
\begin{eqnarray*}
\sum_{p=2m+1}^{\infty}\sum_{i=1}^{\infty}\lambda_{i,n}|\mu_{p,1}^{i,n}|^2\le 2\delta/5.
\end{eqnarray*}
But this inequality contradicts the inequality \eqref{fir1cin}.

The  equalities $\tr\rho_1=\tr\rho_2=\lim_{n\to\infty}\tr\rho^{(n)}$ establishes the last part of the theorem.  
\subsection{Proof of Theorem \ref{XinfH2fin}}\label{subsec:prfqs}
We first observe:
\begin{lemma}\label{conff}  Let $\cH_1,\cH_2$ be two separable Hilbert spaces.
Assume that $\rho_i\in\rT(\cH_i)$ for $i\in[2]$.
Let $\cH=\cH_1\otimes\cH_2$.   Then the function $f:\rT(\cH)\to [0,\infty)$  given by
\eqref{deffuncf} is a convex Lipschitz function with the Lipschitz constant $2$.  Furthermore
\begin{eqnarray}\label{lowbdf}
f(X)\ge 2\|X\|_1-\|\rho_1\|_1-\|\rho_2\|_1 \textrm{ for } X\in\rT_+(\cH).
\end{eqnarray}
\end{lemma}
\begin{proof}Assume that $X_1,X_2\in\rT(\cH)$.
We first show that $f$ is a Lipschitz function with the Lipschitz constant $2$.   Then
\begin{eqnarray*}
&&|f(X_1)-f(X_2)|\\
&&=|\|\tr_2 X_1-\rho_1\|_1-\|\tr_2 X_2-\rho_1\|_1 +\|\tr_1 X_1-\rho_2\|_1-\|\tr_1 X_2-\rho_2\|_1|\\
&&\le|\|\tr_2 X_1-\rho_1\|_1-\|\tr_2 X_2-\rho_1\|_1|+|\|\tr_1 X_1-\rho_2\|_1-\|\tr_1 X_2-\rho_2\|_1|\\
&&\le\|\tr_2(X_1-X_2)\|_1 +\|\tr_1(X_1-X_2)\|_1\le2 \|X_1-X_2\|_1.
\end{eqnarray*}
We now show the convexity of $f$.  Assume that $t\in(0,1)$.  Let $X=tX_1+(1-t)X_2$.  Then
\begin{eqnarray*}
f(X)&=&\|\tr_2(tX_1+(1-t)X_2)-(t+(1-t))\rho_1\|_1\\
&&+\|\tr_1(tX_1+(1-t)X_2)-(t+(1-t))\rho_2\|_1\\
&\le& t\|\tr_2 X_1-\rho_1\|_1+(1-t)\|\tr_2 X_2-\rho_1\|_1\\
&&+t\|\tr_1 X_1-\rho_2\|_1+(1-t)\|\tr_1 X_2-\rho_2\|_1\\
&=&tf(X_1)+(1-t)f(X_2).
\end{eqnarray*}

Assume that $X\in\rT_+(\cH)$.  Then $\tr_j X\in \rT_+(\cH_{j+1})$ for $j\in[2]$, where $\cH_3=\cH_1$.  Hence $\|X\|_1=\tr X=\tr(\tr_j X)=\|\tr_j X\|_1$ for $j\in[2]$.  The triangle inequality yields
\begin{eqnarray*}
f(X)\ge \|\tr_2 X\|_1 -\|\rho_1\|_1 +\|\tr_1 X\|_1-\|\rho_2\|_1=2\|X\|_1-\|\rho_1\|_1-\|\rho_2\|_1.
\end{eqnarray*}
\end{proof}
\begin{lemma}\label{inffrho}  Let the assumptions of Lemma \ref{conff}  hold.
Assume that $\cX\subseteq \cH$ is a closed infinite dimensional subspace with an orthonormal basis $\bx_i, i\in\N$.  Let $\cX_n$ be the subspace spanned by $\bx_1,\ldots,\bx_n$ for $n\in\N$.  Consider the infimum \eqref{defminSX1n}.  Then
\begin{eqnarray}\label{inf=min}\quad\quad
\mu_n(\rho_1,\rho_2)=\min\{f(X), X\in\rS_+(\cX_n), \|X\|_1\le \|\rho_1\|_1+\|\rho_2\|_1\}.
\end{eqnarray}
 Furthermore, the sequence $\mu_n(\rho_1,\rho_2), n\in\N$ is nonincreasing.
\end{lemma}
\begin{proof}  Clearly $f(0)=\|\rho_1\|_1+\|\rho_2\|_1$.  Hence $\mu_n(\rho_1,\rho_2)\le f(0)$.  Suppose that $X\in\rT_+(\cH)$ and $\|X\|_1>f(0)$.  The inequality \eqref{lowbdf} yields that $f(X)\ge 2\|X\|_1-f(0)>f(0)$.  Hence it is enough to consider the infimum \eqref{defminSX1n} restricted to $\{X\in\rS_{+}(\cX_n), \|X\|_1\le f(0)\}$.  This is a  compact finite dimensional set.  Hence the infimum is achieved.
As $\cX_n\subset \cX_{n+1}$ we deduce that $\mu_{n+1}(\rho_1,\rho_2)\le \mu_n(\rho_1,\rho_2)$ for each $n\in\N$.
\end{proof}
\textbf{Proof of Theorem \ref{XinfH2fin}.}
First assume that there exists $\rho\in \rT_+(\cH)$ such that $\tr_2\rho=\rho_1, \tr_1\rho=\rho_2$ and $\supp\rho\subseteq\cX$.  As $\tr\rho=\tr\rho_1$ we deduce that $\rho\in\rS_{+,1}(\cH)$.  Next observe  $\rho\in T_+(\cX)$.  Let $P_n\in B(\cH)$ be the projection on span of $\bx_1,\ldots,\bx_n$.  Then $P_n\in B(\cX)$ and $P_n, n\in\N$ converges in the strong operator topology to $I_{\cX}$.  Lemma 5 in \cite{Fri18} yields that $\lim_{n\to\infty}\|P_n\rho P_n-\rho\|_1=0$ in $\rT(\cX)$.  As $\supp P_n\rho P_n\subseteq\cX_n$ it follows that $P_n \rho P_n\in\rS_+(\cX_n)$ converges to $\rho$ in norm in $T(\cH)$.   Hence $\lim_{n\to\infty} f(P_n\rho P_n)=0$.
Clearly, $\mu_n(\rho_1,\rho_2)\le f(P_n\rho P_n)$.  Hence $\lim_{n\to\infty}\mu_n(\rho_1,\rho_2)=0$.

Second assume that $\lim_{n\to\infty}\mu_n(\rho_1,\rho_2)=0$.  Assume that $\rho^{(n)}\in \rT_+(\cH)$, $\supp\rho^{(n)}\subseteq \cX_n$ and $\mu_n(\rho_1,\rho^{(n)})=f(\rho^{(n)})$.  Clearly
$$\lim_{n\to\infty}\|\rho^{(n)}\|_1=\lim_{n\to\infty}\tr\rho^{(n)}=\|\rho_1\|=\tr\rho_1.$$
Thus the sequence $\rho^{(n)},n\in\N$ is bounded. Hence, there exists a subsequence $\rho^{(n_k)}$ which converges in weak operator topology to $\rho$.  Let $\bx\in \cX^{\perp}$.  Then $\bx\in\cX_n^\perp$.  Therefore $\rho^{(n)}\bx=0$ and $\langle\rho^{(n)}\bx,\by\rangle=0$ for each $\by\in\cH$.  As $\rho^{(n_k)}\overset{w.o.t.}{\to}\rho$ we deduce that $\langle\rho\bx,\by\rangle=0$ for each $\by\in\cH$.  Hence $\rho\bx=\0$.  Thus $\supp \rho\subseteq \cX$.  As $\lim_{k\to\infty}f(\rho^{(n_k)})=0$ Theorem \ref{wotnrmconv} yields that $\lim_{k\to\infty}\|\rho^{(n_k)}-\rho\|_1=0$.  Hence $\tr_2\rho=\rho_1$ and $\tr_1\rho=\rho_2$.\qed
\section{An SDP solution when $\cX$ is finite dimensional}\label{sec:dfcX}
The quantum Strassen problem can be easily generalized to a standard semidefinite  problem in the finite dimensional case.  The feasible set is bounded and contains a positive definite matrix.  Hence we can solve this problem using interior-point methods \cite{NN1994}.  Moreover, the strong duality for this SDP problems holds.  In this section we show that we can extend this approach to separable infinite dimensional $\cH_1$ and $\cH_2$ provided that $\cX$ is finite dimensional.
\subsection{Finite dimensional case }\label{sub:fdgenSDP}
Let $\cH=\cH_1\otimes \cH_2$ be a finite dimensional Hilbert space. Let $\mathcal{X}\subseteq \mathcal{H}$ be a closed subspace. Given two partial density  operators $\rho_{i}\in \rS_{+}(\H_i)\setminus\{0\}$, $i\in[2]$. We now state the following  SDP problem:
\begin{eqnarray}\label{maxSDPrho12}
&&\mu(\rho_1,\rho_2,\cX)=\max\{\tr (XP_{\cX}), \;X\in \rS_+(\cH), \tr_2 X\preceq \rho_1, \tr_1 X\preceq \rho_2\}.
\notag
\end{eqnarray}
Note that the feasible set is convex and bounded, as $\tr X\le\min (\tr\rho_1,\tr\rho_2)$.
If $\supp\rho_i=\cH_i$ for $i\in[2]$ then a feasible set contains a positive definite matrix.   In other cases it is easy to show that it is enough to restrict the problem to $\cH_i'=\supp\rho_i$ for $i\in[2]$ and $\cH'=\cH_1'\otimes \cH_2'$.  Then we can replace $\cX$ by $\cX'=\cX\cap \cH'$.

We write down its primal problem and dual problem.
\begin{displaymath}
	\begin{array}[t]{cc}

\begin{tabular}{c}
\underline{ Primal problem}\\

\begin{tabular}[t]{ll}
maximize:   & $\langle A,X \rangle$, \\

subject to: &

   $\Phi(X)\preceq B;$\\
&  $X\in \rS_+(\mathcal{H}_{1}\otimes\mathcal{H}_{2})$

\end{tabular}

		\end{tabular}

&

\begin{tabular}{c}
\underline{ Dual problem}\\

\begin{tabular}[t]{ll}
minimize:   & $\langle B,Y \rangle$ \\

subject to: &
$\Phi^{*}(Y)\succeq A;$\\
&$Y\in \rS_+(\cH_1\oplus\cH_2)$\\
\end{tabular}

		\end{tabular}
	\end{array}
\end{displaymath}
Here 
\begin{eqnarray*}
&&\Phi:\rS_+(\cH_1\otimes \cH_2)\to \rS_+(\cH_1\oplus \cH_2),  \quad \Phi^*:\rS_+(\cH_1\oplus \cH_2)\to\rS_+(\cH_1\otimes \cH_2)\\ 
&& A=P_{\cX}, \quad B= \begin{bmatrix} \rho_1 & \\ &\rho_2 \end{bmatrix},\nonumber \\
&&   \Phi(X)=\begin{bmatrix} \tr_2(X) & \\ &\tr_1(X) \end{bmatrix}, \nonumber \\
&&  \Phi^*(Y)=\Phi^*\begin{bmatrix} Y_1 &\\ & Y_2 \end{bmatrix}=Y_1\otimes I_2 + I_1 \otimes Y_2. \nonumber
\end{eqnarray*}
It's easy to check the following equality:
\[\forall M,N, \langle \Phi(M),N \rangle=\langle M,\Phi^*(N) \rangle.   \]
Moreover, the strong duality holds for this semidefinite program as we can check that the primal feasible set is not empty, ($0$ is an allowable point),  and there exists an interior point in the dual feasible set.
\begin{itemize}
\item
A primal feasible point: set $X=0 \in \rS_+(\cH_1\otimes\cH_2),  \tr_1(X)\preceq\rho_2,\tr_2(X)\preceq\rho_1$.
\item
A dual strict feasible point: set $Y= I_1\oplus I_2 \in \rS_+(\cH_1\oplus\cH_2), \Phi^*(Y)=2I_{12}\succ P_{\cX}$.
\end{itemize}
Hence, the primal and dual problems have no duality gap and
 the bounded optimal solution of (\ref{maxSDPrho12}) can be computed by interior point methods \cite{NN1994}.

\begin{theorem}\label{finitecase}  Let $\rho_i\in S_{+,1}(\cH_i), i\in[2]$.  Assume that $\cX\subset \cH$. There exists $\rho\in S_{+,1}(\cH), \supp \rho\subseteq \cX$ such that $\tr_2\rho=\rho_1, \tr_1\rho=\rho_2$ if and only if $\mu(\rho_1,\rho_2,\cX)=1$.
\end{theorem}
\begin{proof}
Assume that there exists $\rho\in S_{+,1}(\cH), \supp \rho\subseteq \cX$ such that $\tr_2\rho=\rho_1, \tr_1\rho=\rho_2$. We choose $X=\rho$, so $\tr(\rho P_{\cX})=\tr(\rho)=1$  as $\supp \rho\subseteq \cX$. For every feasible point  $X$, $\tr(XP_{\cX})\leq\tr(X)=\tr(\tr_2(X))\leq\tr(\rho_1)=1$. So $\mu(\rho_1,\rho_2,\cX)=1$.

Assume $\mu(\rho_1,\rho_2,\cX)=1$ and  the maximum is reached by  $X_{max}$.
 Then we have $1=\tr(X_{max}P_{\cX})\leq\tr(X_{max})\leq\tr(\rho_1)=1$, so $\tr(X_{max}P_{\cX})=\tr(X_{max})$, it means that $\supp (X_{max})\subset\cX$.
  From $\tr_2 X\preceq \rho_1$ and $\tr(\rho_1-\tr_2(X_{max}))=0$, we derive $\rho_1=\tr_2(X_{max})$.  In the same way,  we can show  $\rho_2=\tr_1(X_{max})$.
\end{proof}

According to Theorem \ref{finitecase}, we can check the existence of quantum lifting by checking whether $\mu(\rho_1,\rho_2,\cX)$ is equal to  $1$. This can be done numerically by  verifying  if $\mu(\rho_1,\rho_2,\cX)>1 - \varepsilon$ for a given $\varepsilon$ in polynomial time in the given data,  see Nesterov and Nemirovsky \cite{NN1994}.
\subsection{Infinite dimensional case}\label{subsec:infdcase}
In this subsection we assume that $\cX\subset \cH$ is finite dimensional.
\subsubsection{ $\cH_1$ is infinite dimensional and  $\cH_2$ is finite dimensional}\label{subsub:H1fd}
\begin{lemma}\label{H2Xfd}  Let $\mathcal{H}_{1},\mathcal{H}_{2}$ be separable Hilbert spaces of dimensions $N_1=\infty,N_2<\infty$.  Assume that $\cX\subset \cH$ is a finite dimensional subspace of dimension $N$.  Then there exists a finite dimensional subspace $\cH_1'\subset \cH_1$ of dimension $NN_2$ at most such that
$\cX\subset \cH'=\cH_1'\otimes\cH_2$.
\end{lemma}
\begin{proof} Assume that $\be_{i,1}, i\in \N$ is an orthonormal basis in $\cH_1$, and $\cH_2$ has an orthonormal basis
$\{\be_{1,2},\dots,\be_{N_2,2}\}$.  Assume that $\bx_1,\ldots,\bx_N$ is a basis in $\cX$.  Then
\[\bx_l=\sum_{i=p=1}^{\infty, N_2}x_{ip,l} \be_{i,1}\otimes\be_{{p},2}, \quad l\in [N].\]
Set $\bu_{l,p}=\sum_{i=1}^\infty x_{ip,l}\be_{i,1}$.  Then $\bx_l=\sum_{p=1}^{N_2} \bu_{l.p}\otimes \be_{p,2}$.  Let $\cH_1'$ be the subspace of $\cH_1$ spanned by $\bu_{l.p}$ for $l\in [N], p\in [N_2]$.  Then $\dim\H_1'\le N N_2$ and $\cX\subseteq \cH_1'\otimes \cH_2$.
\end{proof}
Thus, in this case  the coupling problem is a finite dimensional problem.
\subsubsection{$\cH_1$ and $\cH_2$ are infinite dimensional}

Assume that $\cH$ is an infinite dimensional separable Hilbert space.  Let $\cX$ be a closed subspace.  Then $\rB(\cX)$ is the subspace of all bounded operators in $L\in \rB(\cH)$ such that $L(\cX)\subseteq\cX$ and $L(\cX^\perp)=0$.  In particular, $L\in \rB(\cH)$ has support in $\cX$ if and only if $L\in \rB(\cX)$.

We assume now that $\cX$ is finite dimensional, and $N=\dim \cX$.
  Then $\rB(\cX)$ has complex dimension $N^2$.  It can be identified with $\C^{N\times N}$ as follows.  Fix an orthonormal basis $\bx_1,\ldots,\bx_N$ in $\cX$.  Then a basis in $\rB(\cX)$ is $\bx_i\bx_j^\vee$ for $i,j\in[N]$.  Thus $L\in \rB(\cX)$ is of the form $L=\sum_{i=j=1}^{N} a_{ij}\bx_i\bx_j^\vee$.  Hece $L$ is one-to-one correspondence with $A=[a_{ij}]\in \C^{N\times N}$.  Observe next that $L\in \rS(\cX)$ if and only if $A$ is Hermitian.

In what follows we need the following lemma:
\begin{lemma}\label{convfdsup}  Let $\cH$ be an infinite dimensional separable Hilbert space.  Assume that $\cX\subset \cH$ is a finite dimensional subspace of dimension $N$.  Assume that $\bx_1,\ldots,\bx_N$ is an orthonormal basis in $\cX$.
Let $Q_n, n\in\N$ be a sequence of projections such that $Q_n\to I$ in the strong operator topology.
Set $\cX_n=Q_n \cX$.
\begin{enumerate}
\item There exists $K\in\N$ such that $\dim \cX_n=N$  for $n >K$.
\item Let $\rho^n\in S_+(\cX_n)$ and assume that $\tr \rho^n\le c$ for $n>K$.  Then
there exists a subsequence $\rho^{n_k}$ that converges in trace norm to $\rho\in S_+(\cX)$.
\end{enumerate}
\end{lemma}
\begin{proof}
First observe that since $Q_n$ is a projection we have the inequality $\|Q_n\bx_i\|\le 1$ for $i\in [N]$ and $n\in\N$.
As $\lim_{n\to\infty}\|Q_n\bx_i-\bx_i\|=0$ for $i\in[N]$ we deduce that for a given $\varepsilon>0$ there exists $K(\varepsilon)$ such that
\begin{eqnarray*}
1-\varepsilon< \langle Q_n\bx_i,Q_n\bx_i\rangle \le 1, \quad |\langle Q_n\bx_i, Q_n\bx_j\rangle|< \varepsilon \textrm{ for }i,j\in[N] \textrm{ and } i\ne j.
\end{eqnarray*}

Let $W_n=[\langle Q_n\bx_i, Q_n\bx_j\rangle]\in \C^{N\times N}$.  Then $W_n$ is Hermitian.  We claim that $W_n$ is positive definite for $\varepsilon < 1/N$.  More precisely $\sigma_1(W_n-I_N)< N\varepsilon$.  (This follows from Perron-Frobenius theorem, as the absolute value of each entry of $I-W_n$ is less than $\varepsilon$.  See \cite{Frb16}.)
Let $\lambda_1(W_n)\ge \cdots\ge \lambda_N(W_n)$ be the eigenvalues of $W_n$.  As $W_n-I_N$ is Hermitian it follows that $|\lambda_i(W_n-I_N)|\le N\varepsilon$.
\noindent
\noindent

(1) For $K=K(1/N)$, $W$ is positive definite.  Hence $Q_n\bx_1,\ldots,Q_n\bx_N$ are linearly independent for $n>K$.

\noindent

(2)  Assume that $n>K$.  Denote by $W_n^{1/2}$ the unique positive definite matrix which is the square root of $W_n$.  Note that the eigenvalues of $W_n^{1/2}$ satisfy also the inequality $|\lambda_i(W_n^{1/2}-I_N)|<N\varepsilon$.  Hence $\lim_{n\to\infty} W_n^{1/2}=I_N$.  Observe that $L\in B(\cX_n)$ is of the form $\sum_{i=j=1}^N a_{ij}Q_n\bx_i(Q_n\bx_j)^\vee$.  Furthermore $\rho\in S_+(\cX_n)$ if and only if $A=[a_{ij}]\in\C^{N\times N}$ is Hermitian and positive semidefinite.  However, the trace of $L$ is not equal to the trace of $A$ but to the trace of $W_n^{-1/2}AW_n^{-1/2}$ which is $\tr W_n^{-1}A$.  This follows from the observation
 that  $\cX_n$ has an orthonormal basis $(\bx_1,\ldots,\bx_N)W^{1/2}$.  Note that
\[(1-N\varepsilon)I_N\preceq W_n \preceq (1+N\varepsilon) I_N \iff (1+N\varepsilon)^{-1}I_N\preceq W_n \preceq (1-N\varepsilon)^{-1} I_N.\]
Hence for $A\succeq 0$ we get
\[(1+N\varepsilon)^{-1}\tr A\le\tr \rho\le  (1-N\varepsilon)^{-1}\tr A\]
Assume that $\rho^n\in S_+(\cX_n)$ is a sequence whose trace is bounded above.
Let $\rho^n=\sum_{i=j=1}^N a_{ij,n}Q_n\bx_i(Q_n\bx_j)^\vee, n>K$.  Set $A_n=[a_{ij,n}]\in\C^{N\times N}$.  Then $A_n,n>K$ are
 positive semidefinite matrices with bounded traces.  Therefore there exists a subsequence $A_{n_k}$ which converges entrywise to $A=[a_{ij}]$.  Set $\rho=\sum_{i=j=1}^N a_{ij}\bx_i\bx_j^\vee$.  It now follows that $\lim_{k\to\infty} \|\rho^{n_k}-\rho\|_1=0$.
\end{proof}

%gjt begin------------------------------
\begin{lemma}\label{rhonlesrho1n}
Let $\cH_1,\cH_2$ be two separable Hilbert spaces with countable orthogonal bases $\be_{i,1}, \be_{i,2}$ for $i\in \N$ respectively. Set $\cH=\cH_1\otimes\cH_2 $. Assume that $\rho\in S_+(\cH),\rho_i \in S_+(\cH_i) $ are given and $\tr_i\rho=\rho_i,  i\in[2] $. Let $P_{n,i}\in S_+(\cH_i)  $ be the orthogonal projection on $\cH_{i,n}=\spa(\be_{1,i},\dots,\be_{n,i})  $. For $n\in\N, i\in[2] $,  $\rho_{i,n}=P_{n,i}\rho_i P_{n,i} $. Let $\rho^{(n)}=(P_{n,1}\otimes P_{n,2})\rho(P_{n,1}\otimes P_{n,2})$. Then we have  $\tr_2\rho^{(n)}\preceq\rho_{1,n}, \tr_1\rho^{(n)}\preceq\rho_{2,n}$.
\end{lemma}
\begin{proof}
Write
\begin{eqnarray*}
&&\rho=\sum_{p,q=1}^{\infty}\rho_{1,pq}\otimes \be_{p,2}\be_{q,2}^\vee=\sum_{i,j=1}^{\infty}\be_{i,1}\be_{j,2}^\vee\otimes\rho_{2,ij},  \\
&&\rho_1=\tr_2\rho=\sum_{p=1}^{\infty}\rho_{1,pp},\quad
\rho_2=\tr_1\rho=\sum_{i=1}^{\infty}\rho_{2,ii}.
\end{eqnarray*}
Where $\rho_{1,pq}$ is in a trace class operator on $\cH_1$ and $\rho_{2,ij}$ is in a trace class operator on $\cH_2$. Then
\begin{eqnarray*}
\tr_2\rho^{(n)}&=&\tr_2\left((P_{n,1}\otimes P_{n,2})(\sum_{p,q=1}^{\infty}\rho_{1,pq}\otimes \be_{p,2}\be_{q,2}^\vee)(P_{n,1}\otimes P_{n,2})\right)\\
&=& \tr_2\left(\sum_{p,q=1}^{n}P_{n,1}\rho_{1,pq}P_{n,1}\otimes \be_{p,2}\be_{q,2}^\vee\right)\\
&=&\sum_{p=1}^{n}P_{n,1}\rho_{1,pp}P_{n,1}\preceq\sum_{p=1}^{\infty}P_{n,1}\rho_{1,pp}P_{n,1}=\rho_{1,n}.
\end{eqnarray*}
Similarly $\tr_1\rho^{(n)}\preceq\rho_{2,n}$.
\end{proof}
%
%With these lemma, we  proved that:
\begin{theorem}\label{H1H2inf}
Let $\cH_1,\cH_2$ be two separable Hilbert spaces with countable orthogonal bases $\be_{i,1}, \be_{i,2}$ for $i\in \N $ respectively. Set $\cH=\cH_1\otimes\cH_2 $. Suppose  $\cX \subset \cH $ is finite  dimensional. Assume that $\rho_i \in \rS_+(\cH_i) $ are given and $\tr\rho_1=\tr\rho_2=1 $. Let $P_{n,i}\in \rS_+(\cH_i)  $ be the orthogonal projection on $\cH_{i,n}=\spa(\be_{1,i},\dots,\be_{n,i})  $. For $n\in\N, i\in[2] $, set $\cX_n=(P_{n,1}\otimes P_{n,2}\cX ) $ and $\rho_{i,n}=P_{n,i}\rho_i P_{n,i} $.

Consider the semidefinite programming  problem
\begin{eqnarray}
\mu_n(\rho_{1},\rho_{2},\cX)=&& \nonumber \\
 max\{ \tr(XP_{\cX_n});&&   \tr_2 X\preceq \rho_{1,n}, \tr_1 X\preceq \rho_{2,n},   \nonumber\\
&&X\in (P_{n,1}\otimes P_{n,2}) \rS_+(\cH)(P_{n,1}\otimes P_{n,2})\} \nonumber
\end{eqnarray}
Then the following statements are equivalent
\begin{enumerate}
\item $\exists \rho\in \rS_{+,1}(\cH_1\otimes\cH_2)$ satisfies
\[ \tr_1(\rho)=\rho_2,\tr_2(\rho)=\rho_1,\supp(\rho)\subset\cX.\]
\item $\lim_{n\to\infty}\mu_n(\rho_{1},\rho_{2},\cX)=1.$
\end{enumerate}
\end{theorem}
\begin{proof}

 $(1) \Rightarrow (2)$ Assume that  there exists an  $ \rho\in \rS_{+,1}(\cH) $  such that  $\tr_2\rho=\rho_1, \tr_1\rho=\rho_2 $, $\supp(\rho)\subset\cX $. Let $\rho^{(n)}=(P_{n,1}\otimes P_{n,2})\rho(P_{n,1}\otimes P_{n,2})$, $\rho^{(n)}\in (P_{n,1}\otimes P_{n,2}) \rS_+(\cH)(P_{n,1}\otimes P_{n,2})$. According to  Lemma \ref{rhonlesrho1n}, we have  $\tr_2\rho^{(n)}\preceq\rho_{1,n}, \tr_1\rho^{(n)}\preceq\rho_{2,n}$. Therefore,  $\rho^{(n)}$ is a feasible solution of the maximal problem. Moreover, since $\supp(\rho)\subset\cX$, we deduce that $\rho^{(n)}(\cH)=(P_{n,1}\otimes P_{n,2})\rho(P_{n,1}\otimes P_{n,2})(\cH)\subset\cX_n$. As $\cX_n$ is closed, and $\supp(\rho^{(n)})$ is the closure of $\rho^{(n)}(\cH)$, we have $\supp(\rho^{(n)})\subset \cX_n $. So we have
 \begin{eqnarray}\label{murho}
\mu_n(\rho_{1},\rho_{2},\cX)\geq\tr(\rho^{(n)}P_{\cX_n})=\tr(\rho^{(n)}).
\end{eqnarray}
 Since $P_{n,1}\otimes P_{n,2}\to I_1\otimes I_2$ in the strong operator topology \cite[Lemma 5]{Fri18}  yields  $\lim_{n\to\infty}\|\rho^{(n)}-\rho\|_1=0$.
 So $\lim_{n\to\infty}\tr\rho^{(n)}=\tr\rho=1$.

 Since  $P_{\cX_n}\preceq I$, and  $X\in S_+(\cH), \tr_2 X\preceq \rho_1$ we obtain
 \begin{eqnarray*}
 \tr XP_{\cX_n}=\tr X^{1/2}X^{1/2}P_{\cX_n}=\tr X^{1/2}P_{\cX_n} X^{1/2}\le\\
 \tr X^{1/2}I X^{1/2}=\tr X=\tr (\tr_2 X)\le \tr\rho_1=1.
 \end{eqnarray*}
 Hence $\tr(\rho^{(n)})=\tr(\rho^{(n)}P_{\cX_n})\le \mu_n(\rho_{1},\rho_{2},\cX)\le 1 $.
 By taking the limit on both side we deduce $\lim_{n\to\infty}\mu_n(\rho_1,\rho_2,\cX)= 1$.

 $(2)\Rightarrow (1)$ Let $\varepsilon_n, n\in\N $ be a positive sequence converging to zero. Suppose that
\[\tr(\rho^{(n)}P_{\cX_n})\geq \mu_n (\rho_{1},\rho_{2},\cX)-\varepsilon_n,   \]
 and $\tr_2 \rho^{(n)}\preceq\rho_{1,n}, \tr_1 \rho^{(n)}\preceq\rho_{2,n}$,and $\rho^{(n)}\in(P_{n,1}\otimes P_{n,2}) S_+(\cH)(P_{n,1}\otimes P_{n,2})\subset S_+(\cX_n)$.
 According to Lemma \ref{convfdsup}(2), there exists $n_k$, such that $\rho^{(n_k)}$ converges in trace norm to $\rho\in S_+(\cX)$.  Lemma \ref{tracecont} yields that $\tr_i\rho^{(n_k)}$ converges to $\tr_i\rho$ in trace norm for $i\in[2]$.
 By taking the limit of the following inequality
 \[\mu_n(\rho_{1},\rho_{2},\cX)-\varepsilon_n\le \tr(P_{\cX_{n_k}}\rho^{(n_k)})\le \tr(\rho^{(n_k)})\le 1.  \]
 We have $\lim_{n\to\infty}\tr(\rho^{(n_k)})=1 $. As $\rho^{(n_k)}$ converges in trace norm to $\rho$, we deduce that $\tr(\rho)=1$.

For each $n_k$, we have
$\tr_i(\rho^{(n_k)})\preceq \rho_{j,n_k}$, where $\{i,j\}=[2]$.  Lemma 5 in \cite{Fri18} yields that $\lim_{k\to\infty} \rho_{j,n_k}=\rho_j$ for $j\in[2]$.  Hence $\tr_i\rho\preceq \rho_j$ for $\{i,j\}=[2]$.   Furthermore, $\tr(\tr_i\rho)=\tr\rho_1=\tr\rho_2=1$.  Hence $\rho_1=\tr_2\rho$ and  $\rho_2=\tr_1\rho$.
\end{proof}
\section{Continuity of the Hausdorff metric}\label{subsec: Hausdorff}
Let $\Phi$ be given by \eqref{defPhiid}.
Lemma \ref{tracecont} yields that $\Phi$ is a bounded linear operator satisfying $\|\Phi\|\le 2$.  Denote $\Sigma=\Phi(\rT_+(\cH_1\otimes\cH_2))$. Note that $(\rho_1,\rho_2)\in \Sigma$ if and only if $\rho_i\in\rT_+(\cH_i)$ and $\tr\rho_1=\tr\rho_2$.
\begin{proposition}\label{Compac}  Assume that $(\rho_1,\rho_2)\in\Sigma$. Then the set $\cM(\rho_1,\rho_2)$ given by \eqref{defcM} is a nonempty, convex, compact, metric set with respect to the distance induced by the norm in $\rT(\cH_1\otimes\cH_2)$. That is for each sequence $\gamma_m\in \cM(\rho_1,\rho_2), m\in\N$ there exists a subsequence $\gamma_{m_k}$ which converges in norm to $\gamma\in \cM(\rho_1,\rho_2)$.
\end{proposition}
\begin{proof}  Clearly  $\cM(0,0)=\{0\}$ and the proposition is trivial in this case.  Assume that $\tr\rho_1=\tr\rho_2>0$.  Then $\frac{1}{\tr \rho_1} \rho_1\otimes \rho_2\in \cM(\rho_1,\rho_2)$.  Clearly $\cM(\rho_1,\rho_2)$ is a convex metric space.  Note that $\|\gamma\|_1=\tr\rho_1$ for each $\gamma \in\cM(\rho_1,\rho_2)$.  Hence $\cM(\rho_1,\rho_2)$ is a bounded set.  Assume that $\gamma_m\in \cM(\rho_1,\rho_2), m\in\N$.  Then there exists a subsequence $\gamma_{m_k}$ which converges in the weak operator topology to $\gamma$.  Clearly $\tr_2\gamma_{m_k}=\rho_1, \tr_1\gamma_{m_k}=\rho_2$.  Theorem \ref{wotnrmconv} yields that $\lim_{m\to\infty}\|\gamma_{m_k}-\gamma\|_1=0$.  Hence $\gamma\in\cM(\rho_1,\rho_2)$.
\end{proof}

Observe that $\rT_+(\cH_1\otimes\cH_2)$ fibers over $\Sigma$:  $\rT_+(\cH_1\otimes\cH_2)=\cup_{(\rho_1,\rho_2)\in \Sigma} \cM(\rho_1,\rho_2)$. We define the distance between two fibers using the Hausdorff metric.  The distance from $\beta\in \rT(\cH_1\otimes\cH_2)$ to $\cM(\rho_1,\rho_2)$ is defined
as
\begin{eqnarray*}
\textrm{dist}(\beta,\cM(\rho_1,\rho_2))=\inf\{\|\beta-\gamma\|_1, \gamma\in\cM(\rho_1,\rho_2)\}.
\end{eqnarray*}
Since $\cM(\rho_1,\rho_2)$ is compact it follows that there exists $\gamma(\beta)\in \cM(\rho_1,\rho_2)$ such that dist$(\beta,\cM(\rho_1,\rho_2))=\|\beta-\gamma(\beta)\|_1$.  Assume that $(\sigma_1,\sigma_2)\in\Sigma$.  Then the semidistance between $\cM(\sigma_1,\sigma_2)$ and $\cM(\rho_1,\rho_2)$ and  is given as
\begin{eqnarray*}
&&\textrm{sd}(\cM(\sigma_1,\sigma_2),\cM(\rho_1,\rho_2))=\\
&&\sup\{\textrm{dist}(\beta,\cM(\rho_1,\rho_2)), \beta\in\cM(\sigma_1,\sigma_2))\}.
\end{eqnarray*}
Since $\cM(\rho_1,\rho_2)$ and $\cM(\sigma_1,\sigma_2)$ are compact it follows
\begin{eqnarray*}
\textrm{sd}(\cM(\sigma_1,\sigma_2),\cM(\rho_1,\rho_2))=\|\beta-\gamma\|_1 \textrm{ for some }\beta\in\cM(\sigma_1,\sigma_2),\gamma\in\cM(\rho_1,\rho_2).
\end{eqnarray*}
Recall that the Hausdorff distance between $\cM(\sigma_1,\sigma_2)$ and $\cM(\rho_1,\rho_2)$ is given by
\begin{eqnarray*}
&&\textrm{hd}(\cM(\sigma_1,\sigma_2),\cM(\rho_1,\rho_2))=\\
&&\max(\textrm{sd}(\cM(\sigma_1,\sigma_2),\cM(\rho_1,\rho_2)),\textrm{sd}(\cM(\rho_1,\rho_2),\cM(\sigma_1,\sigma_2))).
\end{eqnarray*}
\begin{theorem}\label{hmetricfib} The Hausdorff distance on the fibers over $\Sigma$
is a complete metric.  Furthermore the sequence $\cM(\rho_{1,m},\rho_{2,m}), m\in\N$ converges to $\cM(\rho_1,\rho_2)$ in Hausdorff metric if and only if the sequence $(\rho_{1,m},\rho_{2,m}), m\in\N$ converges in norm to $(\rho_1,\rho_2)$.
\end{theorem}
\begin{proof} Since each $\cM(\rho_1,\rho_2)$ is compact it follows that
\begin{eqnarray*}
&&\textrm{hd}(\cM(\sigma_1,\sigma_2),\cM(\rho_1,\rho_2))=0\iff\cM(\sigma_1,\sigma_2)=\cM(\rho_1,\rho_2)\iff\\
&&(\sigma_1,\sigma_2)=(\rho_1,\rho_2).
\end{eqnarray*}
As  $\cM(\sigma_1,\sigma_2)$ and $\cM(\rho_1,\rho_2)$ are compact there exist
$\beta\in\cM(\sigma_1,\sigma_2)$ and $\gamma\in\cM(\rho_1,\rho_2)$ such that $\textrm{hd}(\cM(\sigma_1,\sigma_2),\cM(\rho_1,\rho_2))=\|\beta-\gamma\|_1$.
Lemma \ref{tracecont} yields that $\|\sigma_1-\rho_1\|_1+\|\sigma_2-\rho_2\|_1\le 2 \textrm{hd}(\cM(\sigma_1,\sigma_2),\cM(\rho_1,\rho_2))$.

Assume that the sequence $\cM(\rho_{1,m},\rho_{2,m}), m\in\N$ is a Cauchy sequence in the Hausdorff metric.  Hence the sequence $(\rho_{1,m},\rho_{2,, m})$, $m\in \N$ is a Cauchy sequence in $\Sigma$.   Therefore there exists $(\rho_1,\rho_2)\in\Sigma$ such that $\lim_{m\to\infty}\|\rho_{1,m}-\rho_1\|_1+\|\rho_{2,m}-\rho_2\|_1=0$.

We now show that the sequence $\cM(\rho_{1,m},\rho_{2,m}), m\in\N$ converges to $\cM(\rho_1,\rho_2)$ in the Hausdorff metric.
Since the sequence $(\rho_{1,m},\rho_{2,m})$ is bounded, and each $\cM(\rho_{1,m},\rho_{2,m})$ is compact, it is straightforward to show using Theorem \ref{wotnrmconv} that the sequence
$\textrm{sd}(\cM(\rho_{1,m},\rho_{2,m}),\cM(\rho_1,\rho_2))$ converges to zero.
It is left to show that
\begin{eqnarray*}
\textrm{sd}(\cM(\rho_1,\rho_2),\cM(\rho_{1,m},\rho_{2,m}))=\textrm{dist}(\gamma_m,\cM(\rho_{1,m},\rho_{2,m}))\to 0, \quad \gamma_m\in\cM(\rho_1,\rho_2).
\end{eqnarray*}
Assume to the contrary that the above condition does not hold.  Then there exists $\delta>0$ and a subsequence $\{m_k\}, k\in\N$ such that $\textrm{dist}(\gamma_{m_k},\cM(\rho_{1,m_k},\rho_{2,m_k}))\ge 2\delta$.
As $\cM(\rho_1,\rho_2)$ is compact there exists a subsequence $\{m_{k_l}\}, l\in\N$  and $\gamma\in \cM(\rho_1,\rho_2)$ such that $\lim_{l\to\infty}\|\gamma_{m_{k_l}}-\gamma\|_1=0$.
Hence we can assume that
\[\textrm{dist}(\gamma,\cM(\rho_{1,m_{k_l}},\rho_{2,m_{k_l}}))\ge \delta\]
 for all $l\in\N$.
Without a loss of generality we assume that $m_{k_l}=l$ for $l\in\N$.
We will contradict this statement.

Firstly, we assume that $\rho_{j,m},\rho_j\succ 0$ and all their eigenvalues are simple.
Then there exists orthonormal bases $\{\be_{n,j,m}\},\{\be_{n,j}\},n\in\N$ of $\cH_j$ such that
\begin{eqnarray*}
&&\rho_{j,m}=\sum_{i_j=1}^{\infty} \lambda_{i_j,j,m}\be_{i_j,j,m}\otimes \be_{i_j,j,m}^\vee,\,\lambda_{i_j,j,m}> \lambda_{i_j+1,j,m}>0,  \\
&&\rho_j=\sum_{i_j=1}^{\infty} \lambda_{i_j,j}\be_{i_j,j}\otimes \be_{i_j,j}^\vee, \,\lambda_{i_j,j}> \lambda_{i_j+1,j}>0.
\end{eqnarray*}
Let $P_{n,j,m}$ and  $P_{n,j}$ be the orthogonal projections of $\cH_j$ on \[\cH_{n,j,m}=\rm{span}(\be_{1,j,m},\ldots,\be_{n,j,m}) ~\text{and}~
\cH_{n,j}=\rm{span}(\be_{1,j},\ldots,\be_{n,j,m})\]
 respectively.  Define $\rho_{j,m}^{(n)}=P_{n,j,m}\rho_{j,m} P_{n,j,m}, \rho_{j}^{(n)}=P_{n,j,m}\rho_{j} P_{n,j,m}$.  As \[\lim_{m\to\infty}\|\rho_{j,m}-\rho_j\|_1=0,\] it follows that  $|\lambda_{i_j,j,m}-\lambda_{i_j,j}|\to 0, \|\be_{i_j,j,m}-\be_{i_j,j}\|\to 0, i_j\to\infty$, after we choose the phases (signs) of $\be_{i_j,i,m}$(\cite{FGZ19} Lemma B.6).  Hence for each $n\in\N$
\begin{eqnarray}\label{convprojrhojm}
&&\lim_{m\to\infty} \|P_{n,j,m}-P_{n,j}\|_1=0,\\
&&\lim_{m\to\infty}\|\rho_{j,m}^{(n)}-\rho_{j}^{(n)}\|_1=0 \Rightarrow \lim \|(\rho_{j,m}- \rho_{j,m}^{(n)})-(\rho_j-\rho_{j}^{(n)})\|_1=0.\notag
\end{eqnarray}

Let $\gamma_{n_1,n_2,m}=P_{n_1,1,m}\otimes P_{n_2,2,m}\gamma P_{n_1,1,m}\otimes P_{n_2,2,m}$ and $\gamma_{n_1,n_2}=P_{n_1,1}\otimes P_{n_2,2}\gamma P_{n_1,1}\otimes P_{n_2,2}$.  Then 
$\lim_{m\to\infty} \|\gamma_{n_1,n_2,m}-\gamma_{n_1,n_2}\|_1=0$, and $\lim_{n_1,n_2\to\infty} \|\gamma-  \gamma_{n_1,n_2}\|_1=0$.

The arguments of the proof of Lemma \ref{rhonlesrho1n} yields that $\tr_2\gamma_{n_1,n_2,m}\preceq \rho_{1,m}^{(n_1)}$, and $\tr_1\gamma_{n_1,n_2,m}\preceq \rho_{2,m}^{(n_2)}$.  Hence $\rho_{j,m}-\tr_{j'}\gamma_{n_1,n_2,m}\succeq 0$, where $\{j,j'\}=[2]$.  Define
\begin{eqnarray*}
\sigma_{n_1,n_2,m}=\gamma_{n_1,n_2,m}+\frac{1}{\tr (\rho_{1,m}-\tr_2\gamma_{n_1,n_2,m})}(\rho_{1,m}-\tr_2\gamma_{n_1,n_2,m})\otimes (\rho_{2,m}-\tr_1\gamma_{n_1,n_2,m}).
\end{eqnarray*}
Then $\sigma_{n_1,n_2,m}\in \cM(\rho_{1,m},\rho_{2,m})$ and
\begin{eqnarray*}
&&\|\sigma_{n_1,n_2,m}-\gamma\|_1\le \|\gamma_{n_1,n_2,m}-\gamma_{n_1,n_2}\|_1+\|\gamma_{n_1,n_2}-\gamma\|_1+\|\rho_{1,m}-\rho_1\|_1+\\
&&\|\rho_1-\tr_2\gamma_{n_1,n_2}\|_1+\|\tr_2\gamma_{n_1,n_2}-\tr_2\gamma_{n_1,n_2,m}\|_1.
\end{eqnarray*}
Recall that (Lemma \ref{tracecont})
\begin{eqnarray*}
\|\rho_1-\tr_2\gamma_{n_1,n_2}\|_1\le \|\gamma-\gamma_{n_1,n_2}\|_1, \,
\|\tr_2\gamma_{n_1,n_2}-\tr_2\gamma_{n_1,n_2,m}\|_1\le \|\gamma_{n_1,n_2}-\gamma_{n_1,n_2,m}\|_1.
\end{eqnarray*}
Use \eqref{convprojrhojm} a choice of $n_1,n_2\gg 1$ and corresponding $m\gg 1$ such that
\begin{eqnarray*}
\textrm{dist}(\gamma, \cM(\rho_{1,m},\rho_{2,m}))\le \|\gamma-\sigma_{n_1,n_2,m}\|_1<\delta.
\end{eqnarray*}
This inequality contradicts our assumption and proves that $\cM(\rho_{1,m},\rho_{2,m})$ converges to $\cM(\rho_1,\rho_2)$ in the Hausdorff metric.

We discuss briefly how to modify the above arguments to general $(\rho_{1,m},\rho_{2,m})$ and $(\rho_1,\rho_2)$.  For $\rho_1,\rho_2$ with simple eigenvalues we don't need to modify anything as $\lambda_{n_j,j,m}> \lambda_{n_j+1,j,m}$ for fixed $n_j$ and $m> N_j(n_j)$.  Let us consider now the case where $\rho_1$ and $\rho_2$ are positive definite but may have multiple eigenvalues.  Each eigenvalue must have a finite multiplicity.  Suppose that $\lambda_{n_j,j}>\lambda_{n_j+1,j}$.  Then  $\lambda_{n_j,j,m}>\lambda_{n_j+1,j,m}$
 for $m>N_j(n_j)$.  As $P_{n_j,j}$ is well defined it follows that \eqref{convprojrhojm} holds.

Denote by $\cH'_j$ the closure of the range of $\rho_j$.  It is straightforward to show using Lemma \ref{tr12ab} that $\cM(\rho_1,\rho_2)\subset \rT_+(\cH_1'\otimes\cH_2')$.
Then $P_{n,j}$ are the corresponding projections in $\cH_j'$.  Then for $\lambda_{n_j,j}>\lambda_{n_j+1,j}$, \eqref{convprojrhojm} holds.

Finally, the last part of the theorem that the sequence $\cM(\rho_{1,m},\rho_{2,m}), m\in\N$ converges to $\cM(\rho_1,\rho_2)$ in Hausdorff metric  if the sequence $(\rho_{1,m},\rho_{2,m}), m\in\N$ converges in norm to $(\rho_1,\rho_2)$ follows straightforward from the above arguments.
\end{proof}

%gjt end-------------------------------------
\emph{Acknowledgment:}  The authors thank Professor Luigi Accardi for  helpful comments and suggestions. Dr. Ludovico Lami for comments.
Professor Mingsheng Ying for the reference \cite{Ming}. L. Zhi would like to thank Yuan Feng for the reference \cite{Hsu17} and Ke Ye for pointing out the Banach-Saks theorem in \cite{BanachSaks}.  S. Friedland acknowledges the support of KLMM during his visits to Academy of Mathematics and Systems Science, and Simons collaboration grant for mathematicians. L. Zhi is  supported by the National Key Research Project of China 2018YFA0306702 and the National Natural Science Foundation of China 11571350.

\bibliographystyle{IEEEtran}

\bibliography{gjt}{}

%WS

%
\appendix
\section{Inequalities for singular values of $L\in\rK(\cH)$}\label{appedixA}
In this Appendix we bring some well known inequalities for singular values of $L\in\rK(\cH)$ that we need in this paper, which do not appear explicitly in \cite{RS98}.
It is well known that positive semidefinite compact operator can be treated essentially as positive semidefinite Hermitian matrices, and more general, compact operators can be treated essentially as matrices. Hence we can extend the known inequalities for the eigenvalues of positive semidefinite hermitian matrices and the singular values of matrices as in \cite[Chapter 4]{Frb16} to the eigenvalues  $L\in \rS_+(\cH)\cap\rK(\cH)$ and singular values of $L\in \rK(\cH)$.  In this Appendix we assume that $\cH$ is an infinite dimension separable Hilbert space.
We start with the following known characterization \cite[Lemma 5]{Fri00}:
\begin{lemma}\label{convoy}  Suppose that $L\in \rS_{+}(\cH)\cap\rK(\cH)$.  Let $\bV$ be an $n$-dimensional subspace of $\cH$ with an orthonormal basis  $\bv_1,\ldots,\bv_n$.  Denote by $L(\bV)$ the $n\times n$ hermitian matrix $[\langle L\bv_i,\bv_j\rangle]_{i,j\in[n]}$.  Then $\sigma_i(L(\bV))\le \sigma_i(L)$ for $i\in[n]$ and these inequalities are sharp.
\end{lemma}
We now show the inequalities \eqref{SVDinprod}, where $L\in\rK(\cH),A\in\rB(\cH)$.  Let $M=AL$.  Then $M^\vee M\in \rS_+(\cH)\cap \rK(\cH)$.  Assume that
$\bv_1,\ldots,\bv_n$ are orthonormal eigenvectors of $M^\vee M$ corresponding to the the first largest eigenvalues $\sigma_1^2(M)\ge \cdots\ge \sigma_n^2(M)$.  Let $\bV=$span$(\bv_1,\ldots,\bv_n)$. The Rayleigh principle states:
\begin{eqnarray*}
\sigma_n^2(M)=\min\{\|M\bx\|^2, \bx\in\bV, \|\bx\|=1\}.
\end{eqnarray*}
Next observe that $\|M\bx\|\le \|A\|\|L\bx\|$.  Hence
\begin{eqnarray*}
&&\sigma_n^2(M)\le \|A\|^2\min\{\|L\bx\|^2, \bx\in\bV, \|\bx\|=1\}=\\
&&\|A\|^2\sigma_n((L^\vee L)(\bV))\le  \|A\|^2\sigma_n(L^\vee L)=\|A\|^2\sigma_n^2(L).
\end{eqnarray*}
As $\sigma_n(AL)=\sigma_n(LA)$ we deduce \eqref{SVDinprod}.

Lemma \ref{convoy} yields the following well known convergence result:
 \begin{lemma}\label{approxsv}  Let $L_n,n\in\N$ and  $L$  in $\rK(\cH)$ and assume that $\lim_{n\to\infty}\|L_n-L\|=0$.  Then $\lim_{n\to\infty} \sigma_i(L_n)=\sigma_i(L)$ for each $i\in\N$.
\end{lemma}
(Use $\lim_{n\to\infty} \|L_n L_n^\vee-L L^\vee\|=0$.)

The next lemma is also well known for matrices \cite{RCT77}, and we need it in the proof of Theorem \ref{wotnrmconv}. (See also \cite[Proof of 1.9 Proposition]{Con91}.)
\begin{lemma}\label{tracein}  Let $L\in \rK(\cH)$.  Assume that $\bx_1,\ldots,\bx_n$ and $\by_1,\ldots,\by_n$ are two orthonormal sets of vectors in $\cH$.  Then one has a sharp inequality:
\begin{eqnarray*}
\sum_{i=1}^n |\langle L\bx_i,\by_i\rangle|\le \sum_{i=1}^n \sigma_i(L)
\end{eqnarray*}
for each $n\in\N$.  In particular, if $L\in\rT(\cH)$ and $\{\bx_i\}, \{\by_i\}, i\in\N$ are two orthonormal sequence in $\cH$ then one has sharp inequality
\begin{eqnarray*}
\sum_{i=1}^{\infty} |\langle L\bx_i,\by_i\rangle|\le \|L\|_1.
\end{eqnarray*}
\end{lemma}
\begin{proof}
Assume the SVD decomposition \ref{SVD}.  Let $U\in \rB(\cH)$ a contraction  satisfying $U\bbf_i=\bg_i$ for $i\in\N$.  (We assume that $U\bx=0$ if $\langle \bx,\bbf_i\rangle =0$ for $i\in\N$.)  Then $L=U|L|=U|L|^{1/2} L^{1/2}$.
Hence
\begin{eqnarray*}
&&|\langle L\bx_i,\by_i\rangle|=|\langle |L|^{1/2}\bx_i, |L|^{1/2}U^{\vee}\by_i\rangle|\le \||L|^{1/2}\bx_i|\ \||L|^{1/2}U^{\vee}\by_i\|, \;i\in[n]\Rightarrow\\
&&\sum_{i=1}^n |\langle L\bx_i,\by_i\rangle|\le\big(\sum_{i=1}^n \||L|^{1/2}\bx_i\|^2\big)^{1/2}\big(\sum_{i=1}^n \||L|^{1/2}U^\vee\by_i\|^2\big)^{1/2}.
\end{eqnarray*}
Let $\bV=$span$(\bx_1,\ldots,\bx_n)$.  Use Lemma \ref{convoy} to deduce
\begin{eqnarray*}
\sum_{i=1}^n \| |L|^{1/2}\bx_i\|^2=\sum_{i=1}^n \langle |L|\bx_i,\bx_i\rangle=\tr |L|(\bV)\le \sum_{i=1}^n \sigma_i(|L|)=\sum_{i=1}^n \sigma_i(L).
\end{eqnarray*}
Similarly $\sum_{i=1}^n \| |L|^{1/2}U^\vee\by_i\|^2\le\sum_{i=1}^n \sigma_i(L)$.
This proves the first inequality of the lemma.  By letting $\bx_i=\bbf_i, \by_i=\bg_i$ for $i\in[n]$ we obtain equality in the first inequality of the lemma.  The second inequality and its sharpness follows straightforward from the first inequality and its sharpness.
\end{proof}

Assume that $L\in \rS(\cH)\cap \rK(\cH)$.  Then in SVD decomposition \eqref{SVD}we have that $\bbf_i=\varepsilon_i\bg_i$, where $\varepsilon_i=\pm 1$ for $i\in\N$.
Hence $\varepsilon_i\sigma_i(L)$ is an eigenvalue of $L$ with the corresponding eigenvector $\bg_i$. Suppose furthermore that $L\in\rT(\cH)$.  Then
$\tr L=\sum_{i=1}^\infty \varepsilon_i\sigma_i(L)$ is the sum of the eigenvalues of $L$.
Equality \eqref{traceform} yields
\begin{eqnarray*}
|\tr L|=|\sum_{i=1}^\infty \sigma_i(L)\langle \bg_i, \bbf_i\rangle|\le \sum_{i=1}^\infty \sigma_i(L)|\langle \bg_i, \bbf_i\rangle|\le \sum_{i=1}^\infty \sigma_i(L)=\|L\|_1.
\end{eqnarray*}
Here we used Cauchy-Schwarz inequality $|\langle \bg_i, \bbf_i\rangle|\le \| \bg_i\| \|\bbf_i\|=1$.  Thus equality $|\tr L|=\|L\|_1$ holds if and only if there exists $z\in\C,|z|=1$ such that $z\langle \bg_i, \bbf_i\rangle= \| \bg_i\| \|\bbf_i\|=1$ for each $i$ satisfying $\sigma_i(L)>0$. That is $z\bg_i=\bbf_i$ if $\sigma_i(L)>0$.  Thus $|\tr L|=\|L\|_1$ if and only if $zL\in\rT_+(\cH)$.

We now prove the well known lemma that we used in Section \ref{sec:prilHs}. The proof of the first inequality can also be found in \cite[Theorem 2.3]{CAM67} or  \cite[Proof of 1.11 Theorem]{Con91}, and other parts in \cite[Section I of Chapter I]{Con91}.
\begin{lemma}\label{prodHSop}  Let $L,M\in\rT_2(\cH)$.  Then $L M\in \rT(\cH)$, and
\begin{eqnarray*}
\|LM\|_1\le \|L\|_2\|M\|_2, \; \tr L M=\langle L, M^\vee\rangle, \; \|L L^\vee\|_1=\tr L L^\vee=\|L\|_2^2.
\end{eqnarray*}
\end{lemma}
\begin{proof}
Assume that $L$ has decomposition \eqref{SVD}.    Then $\sigma_i(L L^\vee)=\sigma_i(L)^2$ for $i\in\N$.  Clearly $L L^\vee\in \rT_+(\cH)$ and $\|L L^\vee\|_1=\tr L L^\vee=\langle L,L\rangle$.

For $n\in \N$ let $L_n=\sum_{i=1}^n \sigma_i(L)\bg_i\bbf_i^\vee$.  Then $\|L_n-L\|=\sigma_{n+1}(L)$.  Furthermore $\sigma_i(L_n)=\sigma_i(L)$ for $i\in[n]$ and $\sigma_i(L_n)=0$ for $i>n$.  Similarly, one defines $M_n$ a finite rank operator so that $\|M_n-M\|=\sigma_{n+1}(M)$, and $\sigma_i(M_n)=\sigma_i(M)$ for $i\in[n]$ and $\sigma_i(M_n)=0$ for $i>n$.  Now $L_n$ and $M_n$ can be represented as matrices $A_n,B_n\in\C^{N_n\times N_n}$, such that $L_nM_n$ is represented by the matrix $A_nB_n$.  (We can assume that $N_n=4n$.)  Then $\sigma_i(A_n)=\sigma_i(L_n), \sigma_i(B_n)=\sigma_i(L_n)$ for $n\in[N_n]$.  For a fixed  $l\in[n]$  \cite[Corollary 5.4.8]{Frb16} yields:
\begin{eqnarray*}
\sum_{i=1}^l \sigma_i(L_nM_n)\le \sum_{i=1}^l \sigma_i(L_n)\sigma_i(M_n)=\sum_{i=1}^l \sigma_i(L)\sigma_i(M).
\end{eqnarray*}
As $\lim_{n\to\infty} \|L_n-L\|+\|M_n-M\|=0$ one deduces that $\lim_{n\to\infty}\|L_n M_n-LM\|=0$.  Lemma \ref{approxsv} yields $\sum_{i=1}^l\sigma_i(LM)\le \sum_{i=1}^l \sigma_i(L)\sigma_i(M)$.  The Cauchy-Schwarz inequality implies
\begin{eqnarray*}
\sum_{i=1}^l \sigma_i(L)\sigma_i(M)\le \big(\sum_{i=1}^l\sigma_i^2(L)\big)^{1/2}\big(\sum_{i=1}^l\sigma_i^2(M)\big)^{1/2}\le \|L\|_2\|M\|_2.
\end{eqnarray*}
Hence $\|LM\|_1\le \|L\|_2\|M\|_2$.

It is left to show the equality $\tr LM=\langle L,M^\vee\rangle$.  Again, as for matrices we easily deduce that  $\tr L_nM_n=\langle L_n,M_n^\vee\rangle$.  Observe next
\begin{eqnarray*}
&&\|L_n-L\|_2^2=\sum_{i=n+1}^{\infty} \sigma_i^2(L), \quad \|M_n-M\|_2^2=\sum_{i=n+1}^{\infty} \sigma_i^2(M) \Rightarrow\\
&&\lim_{n\to\infty} \|L_n-L\|_2+\|M_n-M\|_2=0.
\end{eqnarray*}
Then
\begin{eqnarray*}
&&|\tr (L_n M_n-LM)|=|\tr((L_n-L)M_n+L(M_n-M))|\le \\
&&|\tr((L_n-L)M_n| +|\tr L(M_n-M)|\le \|(L_n-L)M_n\|_1 +\\
&&\|L(M_n-M)\|_1\le\|L_n-L\|_2\|M_n\|_2+\|L\|_2\|M_n-M\|_2\le\\
&&\|L_n-L\|_2\|M\|_2+\|L\|_2\|M_n-M\|_2\Rightarrow \lim_{n\to\infty} \tr L_n M_n=\tr LM.
\end{eqnarray*}
SImilarly
\begin{eqnarray*}
&&|\langle L_n,M_n^\vee\rangle -\langle L,M^\vee\rangle |=|\langle L_n-L,M_n^\vee\rangle +\langle L,M_n-M^\vee\rangle|\le\\
&&|\langle L_n-L,M_n^\vee\rangle| +|\langle L,M_n-M^\vee\rangle|\le\|L_n-L\|_2\|M_n\|_2+\\
&&\|L\|_2\|M_n-M_2\|_2\Rightarrow \lim_{n\to\infty} \langle L_n,M_n^\vee\rangle=\langle L,M^\vee\rangle.
\end{eqnarray*}
As $\tr L_nM_n=\langle L_n,M_n^\vee\rangle$ for $n\in\N$ we deduce the equality  $\tr LM=\langle L,M^\vee\rangle$.
\end{proof}
\section{Convergence in various topologies in $\rT_p(\cH)$}\label{app:conv}
 Let $\cB$ be a Banach space over $\F\in\{\C,\R\}$ with the norm $\|\cdot\|$.  Denote by $\cB^\vee$  the dual Banach space of bounded linear functionals.  Recall that if $\cB=\cH$, then $\cH^\vee$ is identified with $\cH$: Namely, each bounded linear functional on $\cH$ is $\by^\vee$, where $\by\in\cH$.  Namely, $\by^\vee(\bx)=\langle \bx,\by\rangle$.
 Denote by $\rB(\cB)$ the space of bounded linear transformatons of $\cB$ to itself.

 Recall the following convergence notions in $\cB$ and $\rB(\cB)$:
 Let $\{\bx_n\}\subset \cB$ and $\{T_n\}\subset \rB(\cB)$ be given.
 \begin{enumerate}
 \item Convergence in norm: $\bx_n\to\bx$ and $T_n\to T$ if $\lim_{n\to\infty}\|\bx_n-\bx\|=0$ and $\lim_{n\to\infty} ||T_n -T\|=0$ respectively.
 \item Weak convergence : $\bx_n\overset{w}\to\bx$ if for each $\bbf\in\cB^\vee$, the equality $\lim_{n\to\infty} \bbf(\bx_n) = \bbf(\bx)$ holds.
 \item Assume that $\cB=\cB_1^\vee$ for some Banach space $\cB_1$.  Then $\bx_n\overset{w^*}{\to}\bx$ if for each $\by\in\rB_1$ the equality $\lim_{n\to\infty} \bx_n(\by)=\bx(\by)$.
 \item Pointwise convergence: $T_n\overset{pw}{\to} T$ if for each $\bx\in \cB$ the equality $\lim_{n\to\infty} \|T_n\bx-T\bx\|=0$ holds.
 \item Weak operator convergence $T_{n}\overset{w.o.t.}{\to}T$ if for each $\bx\in\cB,\bbf\in \cB^\vee$ equality $\lim_{n\to\infty}\bbf(T_n(\bx))=\bbf(T(\bx))$ holds.
  \end{enumerate}

It is a consequence of Banach-Steinhaus theorem on uniform boundedness, e.g., \cite[Sec. 19, 20, 22, 28]{Sem}, that all the above convergences yield that the sequences $\{\bx_n\}, \{T_n\}, n\in \N$ are uniformly bounded.

In this section we assume that $\cH$ be an infinite dimensional separable space.
Assume that $\be_i, i\in\N$ is an orthonormal basis in $\cH$.  The following lemma is well known,   and we bring the proof of part (3), (which is less known, see   \cite[Lemma 3.1]{CAM67} and \cite[Lemma 19.2]{PBEB15})  for completeness:
\begin{lemma}\label{wnormconvH}
Let $\bx_n\in\cH, n\in\N$.  Then
\begin{enumerate}
\item The sequence $\{\bx_n\},n\in\N$ converges weakly to $\bx\in\cH$ if and only if
\begin{eqnarray*}\label{wconvcond}
\{\|\bx_n\|\},n\in\N \textrm{ is bounded, and } \lim_{n\to\infty} \langle \bx_n,\be_i\rangle=
\langle \bx,\be_i\rangle \textrm{ for }i\in\N.
\end{eqnarray*}
\item  Let $\{\bx_n\},n\in\N$ be a bounded sequence.  There exists a subsequence $\{x_{n_k}\},k\in\N$ which converges weakly to some $\bx\in\cH$.
\item  Suppose that $\bx_n\overset{w}{\to}\bx$.  Then $\liminf \|\bx_n\|\ge \|\bx\|$.  Furthermore,
$\lim_{n\to\infty}\|\bx_n-\bx\|=0$ if and only if $\lim_{n\to\infty} \|\bx_n\|=\|\bx\|$.
\end{enumerate}
\end{lemma}
\begin{proof}
(3) Assume that $\bx_n\overset{w}{\to}\bx$.  Suppose that $M\in\N$ is fixed and let
\begin{eqnarray*}
\bx_n=\sum_{i=1}^\infty x_{i,n}\be_i, \;\bx_{n,M}=\sum_{i=1}^M x_{i,n}\be_i,\;
\bx=\sum_{i=1}^\infty x_{i}\be_i, \;\bx_{M}=\sum_{i=1}^M x_{i}\be_i.
\end{eqnarray*}
Recall that $\lim_{n\to\infty} x_{i,n}=x_i $ for $i\in\N$.
Clearly, $\|\bx_n\|\ge \|\bx_{n,M}\|$.  Hence $\liminf_{n\to\infty}\|\bx_n\|\ge \|\bx_M\|$.  As $M\in\N$ was arbitrary we deduce that

\noindent
$\liminf_{n\to\infty}\|\bx_n\|\ge\|\bx\|$.

Assume in addition that $\lim_{n\to\infty}\|\bx_n\|=\|\bx\|$.
If $\bx=0$ we immediately deduce that $\bx_n\to \0$.
 Assume that $\|\bx\|>0$. Fix $\varepsilon \in (0,1)$.  Then there exists $N(\varepsilon)$ such that $\|\bx_n\|^2< (1+\varepsilon^2/2)\|\bx\|^2$ for $n>N(\varepsilon)$.   Furthermore, there exists $M(\varepsilon)\in\N$ such that for $M>M(\varepsilon)$ $\|\bx_M\|> (1-\varepsilon^2/8)\|\bx\|$.  Hence $\|\bx_M\|^2> (1-\varepsilon^2/4)\|\bx\|^2$, and
 \begin{eqnarray*}
 \|\bx_M-\bx\|^2=\|\bx\|^2-\|\bx_M\|^2<(\varepsilon^2/4)\|\bx\|^2\Rightarrow\|\bx_M-\bx\|<(\varepsilon/2)\|\bx\|.
 \end{eqnarray*}
 Fix $M>M(\varepsilon)$.  Then there exists $N_1(\varepsilon)$ such that for $n>N_1(\varepsilon)$ the inequality $\|\bx_{n,M}-\bx_M\|<(\varepsilon^2/8)\|\bx\|$ holds.
Thus for $n>\max(N(\varepsilon),N_1(\varepsilon))$ we obtain:
\begin{eqnarray*}
(\varepsilon^2/8)\|\bx\|\ge \|\bx_{n,M}-\bx_M\|\ge \|\bx_{M}\|-\|\bx_{n,M}\|\ge (1-\varepsilon^2/8)\|\bx\|-\|\bx_{n,M}\|\Rightarrow\\
 \end{eqnarray*}
Hence
\begin{eqnarray*}
(1+\varepsilon^2/2)\|\bx\|^2\ge \|\bx_n\|^2=\|\bx_{n.M}\|^2 +\|\bx_n-\bx_{n,M}\|^2\ge\\ (1-\varepsilon^2/2)\|\bx\|^2+\|\bx_n-\bx_{n,M}\|^2\Rightarrow \varepsilon^2\|\bx\|^2\ge \|\bx_n-\bx_{n,M}\|^2
\end{eqnarray*}
 Thus,  for $n>\max(N(\varepsilon),N_1(\varepsilon))$ we showed:
 \begin{eqnarray*}
 \|\bx_n -\bx\|\le \|\bx_{n,M}-\bx_{M}\|+\|(\bx_n-\bx_{n,M})-(\bx-\bx_{M})\|\le \\
 \|\bx_{n,M}-\bx_{M}\|+\|(\bx_n-\bx_{n,M})\|+\|(\bx-\bx_{M})\|\le\\
 (\varepsilon^2/8)\|\bx\| +\varepsilon\|\bx\| +(\varepsilon/2)\|\bx\|<2\varepsilon\|\bx\|.
 \end{eqnarray*}
 That is, $\lim_{n\to\infty}\|\bx_n-\bx\|=0$.

 Vice versa, assume that $\lim_{n\to\infty} \|\bx_n-\bx\|=0$.
 Use triangle inequality to deduce that $\lim_{n\to\infty}\|\bx_n\|=\|\bx\|$.
 \end{proof}

We recall the Banach-Saks theorem \cite{BanachSaks}
\begin{theorem}\label{banachsaks}  Suppose that sequence $\{\bx_n\},n\in\N$ converges weakly to $\bx\in\cH$
Then there exists a subsequence $\bx_{n_j},j\in\N$ such that the sequence of arithmetic means of this subsequence converges strongly to $\bx$, i.e.,
\begin{eqnarray*}
\lim_{m\to\infty}\|\frac{1}{m}\sum_{j=1}^{m}\bx_{n_j}-\bx\|.
\end{eqnarray*}
\end{theorem}

We now discuss various topologies on $\rT(\cH)$.  First recall \cite[Theorem VI.26]{RS98} that $\rT(\cH)=(\rK(\cH))^\vee$ and $\rB(\cH)=(\rT(\cH))^\vee$, where
\begin{eqnarray*}
A\mapsto \tr A \rho, \;A\in \rK(\cH), \rho\in \rT(\cH),\\
\rho\mapsto \tr \rho B,\;  \rho\in \rT(\cH), B\in \rB(\cH),
\end{eqnarray*}
are the corresponding linear operators on $\rK(\cH)$ and $\rT(\cH)$.

We next observe the well known result \cite[2.5 Proposition]{Con91} and outline its proof
\begin{lemma}\label{w*=wot}

\noindent
\begin{enumerate}
\item
On the bounded subsets of $\rT(\cH)$ the $w^*$ topology is the weak operator topology.
\item
On the bounded subsets of $\rT_2(\cH)$ the weak operator topology is the weak topology.
\end{enumerate}
\end{lemma}
\begin{proof} (1) First suppose that $\rho_n\overset{w^*}{\to}\rho$ in $\rT(\cH)$.  Then $\lim_{n\to\infty}\tr A\rho_n=\tr A \rho$ for each $A\in\rK(\cH)$.  Assume that $A$ is a rank one operator: $A=\bx\by^\vee$.  Then $\tr A\rho=\tr \rho A=\langle \rho\bx,\by\rangle$.   Hence $w^*$ convergence yields weak operator convergence in $\rT(\cH)$.  Second suppose that $\rho_n\overset{w.o.t.}{\to}\rho$.  Then $\lim_{n\to\infty}\tr A\rho_n=\tr A \rho$ for each rank one operator.  Hence this equality holds for each finite rank operator.  Since each compact $A$ is approximated by a finite rank operator in the operator norm on $\rK(\cH)$ it follows that the convergence in w.o.t. yield the convergence in $w^*$ topology.

(2) The proof of (2) is similar to the proof of (1).
\end{proof}

As the $\rK(\cH)$, viewed as a metric space with respect to the distance $d(A,B)=\|A-B\|$, is a complete separable space, it follows that every bounded sequence $\{\rho_n\}, n\in\N$ in $\rT_1(\cH)$ has a convergent subsequence $\rho_{n_k}\overset{w^*}{\to}\rho$.  We give a constructive version of this result using SVD decomposition and Lemma \ref{wnormconvH}:
\begin{lemma}\label{weaktoplemma}  Suppose that $A_n\in \rT(\cH)$ and $\|A_n\|_1\le K$ for $n\in\N$.  Then there exists a subsequence $\{n_k\}$ such that $A_{n_k}\overset{w.o.t.}{\to} A\in \rT_1(\cH)$.  Furthermore $\|A\|_1\le \liminf\|A_{n_k}\|_1$.  Assume in addition that $A_n\in\rT_+(\cH)$ for $n\in\N$.  Then
\begin{eqnarray*}
A\in \rT_+(\cH),\; \|A\|_1=\tr A\le \liminf_{n_k}\tr A_{n_k}.
\end{eqnarray*}
\end{lemma}
\begin{proof}   Write down the singular value decomposition for each $A_n$:
\[A_n=\sum_{i=1}^\infty \sigma_i(A_n)\bg_{i,n}\bbf_{i,n}^\vee.\]
  Recall that $\{\sigma_i(A_n)\}$ is a nonincreasing nonnegative sequence such that $\sum_{i=1}^\infty \sigma_i(A_n)\le K$ for each $n\in\N$.  Furthermore, the two sets $\{\bg_{i,n}\},\{\bbf_{i,n}\}, i\in \N$ are orthonormal sets of vectors in $\cH$.  Use the Cantor diagonal principle to construct a subsequence $n_k,k\in\N$ such that
\[\lim_{k\to\infty}\sigma_i(A_{n_k})=\sigma_i,\;\bg_{i,n_k}\overset{w}{\to} \bg_i,\; \bbf_{i,n_k}\overset{w}{\to} \bbf_i \textrm{ for all } i\in\N.\]
Clearly, $\{\sigma_i\}$ is a nonnegative nonincreasing sequence such that for each $N\in\N$ one has
\[\sum_{i=1}^N \sigma_i  =\lim_{k\to\infty} \sum_{i=1}^N \sigma_{i}(A_{n_k})\le K.\]
Hence $\sum_{i=1}^\infty \sigma_i \le K$ and  $\lim_{i\to \infty}\sigma_i=0$. %

Let
$D_m=\sum_{i=1}^m \sigma_i \bg_i\bbf_i^\vee$ be a finite rank operator for $m\in\N$.  Hence $D_m\in\rT(\cH)$.
Then for $p>m$ we have that
\begin{eqnarray*}
\|D_p-D_m\|_1\le \sum_{i=m+1}^p \sigma_i\|\bg_i\| \|\bbf_i\|\le \sum_{i=m+1}^\infty \sigma_i\|\bg_i\| \|\bbf_i\|\le \sum_{i=m+1}^\infty \sigma_i.
\end{eqnarray*}
Thus $\{D_m\}$ is a Cauchy sequence in $\rT(\cH)$ which converges to $A=\sum_{i=1}^\infty \sigma_i \bg_i\bbf_i^\vee$, such that
\[\|A\|_1\le \sum_{i=1}^{\infty} \|\sigma_i\bg_i\bbf_i^\vee\|_1=\sum_{i=1}^{\infty}\sigma_i\|\bg_i\| \|\bbf_i\|\le \sum_{i=1}^\infty \sigma_i\le K. \]
Clearly, $\|A\|_1\le \liminf\|A_{n_k}\|_1$.

It is left to show that $A_{n_k}\overset{w.o.t.}{\to} A$.
Assume  for simplicity of the exposition of the following two assumptions.  First,  $n_k=k $ for $k\in\N$.  Second, the given $\bx,\by\in\cH$ satisfy $\|\bx\|,\|\by\|\le 1$.  To show that $A_k\overset{w.o.t.}{\to} A$ it is enough to show the following:
Let $\varepsilon>0$ be given.  Then there exists $K(\varepsilon)=K(\varepsilon,\bx,\by)\in\N$ such that for $k>K(\varepsilon)$ we have $|\langle (A_k-A)\bx,\by\rangle|<3\varepsilon$.

As $\sigma_k\ge 0, k\in\N$ and $\sum_{k=1}^\infty \sigma_k\le K$, there exists $N\in\N$ such that $\sum_{k=N}^\infty \sigma_k<\varepsilon$.  In particular, $\sigma_N<\varepsilon$.  We now let
\begin{eqnarray*}
&&B_k=\sum_{i=1}^N\sigma_i(A_k)\bg_{i,k}\bbf_{i,k}^\vee, \quad C_k=\sum_{i=N+1}^\infty\sigma_i(A_k)\bg_{i,k}\bbf_{i,k}^\vee,\\
&&B=\sum_{i=1}^N\sigma_i\bg_{i}\bbf_{i}^\vee, \quad C=\sum_{i=N+1}^\infty\sigma_i\bg_{i}\bbf_{i}^\vee
\end{eqnarray*}
Thus $A_k=B_k+C_k$ and $A=B+C$.  Clearly, $B_k\overset{w.o.t.}{\to} B$.  Hence, there exists $K_1(\varepsilon)=K_1(\varepsilon,\bx,\by)\in\N$ such that for $k>K_1(\varepsilon)$ one has $|\langle (B_k-B)\bx,\by\rangle|<\varepsilon$.  As $\lim_{k\to\infty}\sigma_N(A_k)=\sigma_N<\varepsilon$ it follows that there exists $K_2(\varepsilon)$ such that for $k>K_2(\varepsilon)$ $\sigma_N(A_k)< \varepsilon$.
As $\sigma_i(A_k),i\in\N$ is a nonincreasing sequence, it follows that $\sigma_i(A_k)<\varepsilon$ for $i\ge N$ and $k>K_2(\varepsilon)$.   Note that the above  expansion of $C_k$ is the SVD expansion of $C_k$ it follows that $\|C_k\|=\sigma_{N+1}(A_k)$.  Hence $\|C_k\|<\varepsilon$ for $k>K_2(\varepsilon)$.  Therefore
$|\langle C_k\bx,\by\rangle|\le \|C_k\|\|\bx\|\by\|<\varepsilon$ for $k>K_2(\varepsilon)$.
Observe next
\[\|C\|\le \sum_{i=N+1}^{\infty} \|\sigma_i\bg_i\bbf_i^\vee\|\le \sum_{i=N+1}^\infty \sigma_i<\varepsilon.\]
Hence $|\langle C\bx,\by\rangle|<\varepsilon$.  Set $K(\varepsilon)=\max(K_1(\varepsilon),K_2(\varepsilon))$.  Then
 \[|\langle (A_k-A)\bx,\by\rangle|\le |\langle (B_k-B)\bx,\by\rangle| +|\langle C_k\bx,\by\rangle|+|\langle C_k\bx,\by\rangle|<3\varepsilon.\]

 Assume in addition that $A_n\in\rT_+(\cH)$ for $n\in\N$.  Then $\bbf_{i,n}=\bg_{i,n}$ for $i,n\in\N$.  Clearly $\|A_n\|_1=\tr A_n$ for $n\in\N$.  Observe next that $A=\sum_{i=1}^n \sigma_i\bg_i\bg_i^\vee\in \rS(\cH)$.  Hence $\langle A\bx,\bx\rangle=\sum_{i=1}^\infty \sigma_i|\langle\bx,\bg_i\rangle|^2$. Therefore $A\in\rT_+(\cH)$.  Thus $\|A\|=\tr A\le \liminf_{n_k} \tr A_{n_k}$.
\end{proof}

We now consider the following simple example: $A_n=\be_n\be_n^\vee\in\rT(\cH), n\in\N$.  Observe that $\sigma_1(A_n)=1$  and $\sigma_i(A_n)=0$ for $i>1$.  Thus $\|A_n\|_1=1$ for $n\in\N$.
Clearly,  $A_n\overset{w.o.t.}{\to}0$.  Hence $A_n\overset{w^*}{\to}0$.  However the sequence $\{A_n\}\subset \rT_1(\cH)$ does not converge to $0$ in the weak topology on $\rT(\cH)$.  Indeed, take the linear functional $A\mapsto \tr A I$, where $I\in\rB(\cH)$ is the identity operator.  Then $\tr A_n I=1$ for $n\in\N$.

We now bring an analog of part (3) of Lemma \ref{wnormconvH} which is due to Davies\cite[Lemma 4.3]{Dav69}:
\begin{lemma}\label{weaknormconv1}
Let $A_n\in\rT_{+}(\cH),n\in\N$ and assume that $A_n\overset{w.o.t.}{\to}A\in\rT(\cH)$.  Then
\begin{eqnarray*}
\lim_{n\to\infty} \|A_n -A\|_1= 0 \iff \lim_{n\to\infty} \tr A_n=\tr A.
\end{eqnarray*}
\end{lemma}
\begin{proof}  Lemma \ref{weaktoplemma} yields that $A\in\rT_{+}(\cH)$ and $\|A\|_1=\tr A$.  Hence
\begin{eqnarray*}
\lim_{n\to\infty} \|A_n -A\|_1=0 \Rightarrow \lim_{n\to\infty} \|A_n\|_1=\|A\|_1\Rightarrow \lim_{n\to\infty} \tr A_n=\tr A.
\end{eqnarray*}

Assume now $\lim_{n\to\infty} \tr A_n=\tr A$.  Assume to the contrary that the sequence $\{A_n\}$ does not converge  to $A$ in norm in $\rT_1(\cH)$.  Hence there exists $\varepsilon_0>0$ and subsequence
$\{A_{m_k}\}$ such that $\|A_{m_k} -A\|_1\ge \varepsilon_0 $ for all $k\in\N$.  To show a contradiction we can assume without loss of generality that $m_k=k, k\in\N$.
Assume that each $A_n$ has the following spectral decomposition:
\begin{eqnarray*}
A_n=\sum_{i=1}^\infty \sigma_i(A_n)\bg_{i,n}\bg_{i,n}^\vee, \quad \langle \bg_i,\bg_j\rangle =\delta_{ij}, i,j\in\N.
\end{eqnarray*}
As in the proof of Lemma \ref{weaktoplemma} there exists
a subsequence $n_k,k\in\N$ such that
\begin{eqnarray*}
&&\lim_{k\to\infty}\sigma_i(A_{n_k})=\sigma_i,\;\bg_{i,n_k}\overset{w}{\to} \bg_i,\;  \textrm{ for all } i\in\N,\\
&&A=\sum_{i=1}^\infty \sigma_i\bg_i\bg_i^\vee.
\end{eqnarray*}
Here $\{\sigma_i\}$ is a nonnegative nonincreasing sequence with $\sum_{i=1}^\infty \sigma_i\le \liminf \|A_{n_k}\|=\tr A$.  As $\|\bg_i\|\le 1, i\in\N$, from the arguments of the proof of Lemma \ref{weaktoplemma}  it follows that
$$\tr A=\sum_{i=1}^\infty \sigma_i \|\bg_i\|^2\le \sum_{i=1}^\infty \sigma_i\le \tr A.$$
Hence for each $\sigma_i>0$ we deduce that $\|\bg_i\|=1$.  Use part (3) of Lemma \ref{wnormconvH} to deduce that
$\lim_{n\to\infty} \|\bg_{i,n}-\bg_i\|=0$ for each $\sigma_i>0$.

Let us now assume the more difficult case: $\sigma_i>0$ for $i\in\N$.  Fix $\varepsilon \in (0,1/8)$.  Then there exists $M(\varepsilon)$ so that $$\sum_{i=1}^{M(\varepsilon)}\sigma_i \ge (1-\varepsilon)\sum_{i=1}^{\infty}\sigma_i =(1-\varepsilon)\tr A.$$
Hence $\sum_{M(\varepsilon)+1}^\infty \sigma_1\le\varepsilon \tr A$.
The assumption that $\lim_{n\to\infty} \tr A_n=\tr A$ yields that there exists $N(\varepsilon)$ such that for $n>N(\varepsilon)$ the inequality $\tr A_n\le (1+\varepsilon)\tr A$ holds.  As $\lim_{k\to\infty}\sigma_i(A_{n_k})=\sigma_i$ for $i\in\N$ we deduce that there exists $N_1(\varepsilon)$ such that for $k>N_1(\varepsilon)$ the inequality
\begin{eqnarray*}
|\sum_{i=1}^{M(\varepsilon)} \sigma_i(A_{n_k})-\sum_{i=1}^{M(\varepsilon)} \sigma_i|\le \sum_{i=1}^{M(\varepsilon)}|\sigma_i(A_{n_k})-  \sigma_i|\le \varepsilon \tr A
\end{eqnarray*}
holds.  Hence
\begin{eqnarray*}
&&\sum_{i=1}^{M(\varepsilon)} \sigma_i(A_{n_k})\ge (1-2\varepsilon)\tr A \textrm{ for } k>N_1(\varepsilon),\\
&&\sum_{i=M(\varepsilon)+1}^\infty \sigma_i(A_{n_k})\le 3\varepsilon\tr A  \textrm{ for } k>\max(N(\varepsilon),N_1(\varepsilon)).
\end{eqnarray*}
We now estimate from above $\|A_{n_k}-A\|_1$ for $n>\max(N(\varepsilon),N_1(\varepsilon))$:
\begin{eqnarray*}
&&\|A_{n_k}-A\|_1=\|\sum_{i=1}^\infty \sigma_i(A_{n_k})\bg_{i,n_k}\bg_{i,n_k}^\vee -\sum_{i=1}^\infty\sigma_i\bg_i\bg_i^\vee\|_1\le\\
&&\sum_{i=1}^{M(\varepsilon)}\|\sigma_i(A_{n_k})\bg_{i,n_k}\bg_{i,n_k}^\vee-\sigma_i\bg_i\bg_i^\vee\|_1 +\\ &&\sum_{i=M(\varepsilon)+1}^\infty \|\sigma_i(A_{n_k})\bg_{i,n_k}\bg_{i,n_k}^\vee\|_1 +\sum_{i=M(\varepsilon)+1}^\infty\|\sigma_i\bg_i\bg_i^\vee\|_1\le\\
&&\sum_{i=1}^{M(\varepsilon)}\|\sigma_i(A_{n_k})\bg_{i,n_k}\bg_{i,n_k}^\vee-\sigma_i\bg_i\bg_i^\vee\|_1+\sum_{i=M(\varepsilon)+1}^\infty\sigma_i(A_{n_k})+
\sum_{i=M(\varepsilon)+1}^\infty\sigma_i\le\\
&&\sum_{i=1}^{M(\varepsilon)}\|\sigma_i(A_{n_k})\bg_{i,n_k}\bg_{i,n_k}^\vee-\sigma_i\bg_i\bg_i^\vee\|_1+4\varepsilon \tr A.
\end{eqnarray*}
We claim that for each $i\in\N$ the equality
$$\lim_{k\to\infty} \|\sigma_i(A_{n_k})\bg_{i,n_k}\bg_{i,n_k}^\vee-\sigma_i\bg_i\bg_i^\vee\|_1=0$$
holds. Write down
\begin{eqnarray*}
&&\sigma_i(A_{n_k})\bg_{i,n_k}\bg_{i,n_k}^\vee-\sigma_i\bg_i\bg_i^\vee=\\
&&(\sigma_i(A_{n_k})-\sigma_i)\bg_{i,n_k}\bg_{i,n_k}^\vee +
\sigma_i(A)(\bg_{i,n_k}-\bg_i)\bg_{i,n}^\vee+\sigma_i(A)\bg_i(\bg_{i,n_k}^\vee-\bg_i^\vee)
\end{eqnarray*}
We now claim that each of the above summands converges to $0$ in $\|\cdot\|_1$ norm.  First recall that $\|\bx\by^\vee\|_1=\|\bx\|\|\by\|$.  Second
\begin{eqnarray*}
\lim_{k\to\infty}|\sigma_i(A_{n_k})-\sigma_i|=0, \quad \lim_{k\to\infty}\|\bg_{i,n_k}-\bg_i\|=0 \textrm{ for all } i\in\N.
\end{eqnarray*}
Hence there exists $N_2(\varepsilon)>\max(N(\varepsilon),N_1(\varepsilon))$ such that for $k>N_2(\varepsilon)$ the inequality
\begin{eqnarray*}
\sum_{i=1}^{M(\varepsilon)}\|\sigma_i(A_{n_k})\bg_{i,n_k}\bg_{i,n_k}^\vee-\sigma_i\bg_i\bg_i^\vee\|_1\le \varepsilon \tr A.
\end{eqnarray*}
Combine all the above inequalities to deduce that $\|A_{n_k}-A\|\le 5\varepsilon\tr A$ for $k>N_2(\varepsilon)$.  Choose $\varepsilon<\frac{\varepsilon_0}{5\tr A}$ to contradict our assumption that $\|A_k-A\|_1\ge \varepsilon_0$ for $k\in\N$.
\end{proof}

We now analyze the norm convergence in $\rT_2(\cH)$.  We believe that most of the results stated in the lemma below are known to the experts.  This lemma is used in the proof of Theorem \ref{wotnrmconv}.  For a closed subspace $\bU\subset\cH$ denote by $P(\bU)\in \rK(\cH)$ the orthogonal projection on $\bU$.
\begin{lemma}\label{convT2} Let $A,B,\in \rT_2(\cH)$, and assume that $\lim_{n\to\infty}\|A_n-A\|_2=0$.  Then
\begin{enumerate}
\item
\begin{eqnarray*}
\sum_{i=1}^\infty |\sigma_i(A)-\sigma_i(B)|^2\le \|A-B\|_2^2.
\end{eqnarray*}
\item Assume that
\begin{eqnarray*}
&&A=\sum_{i=1}\sigma_i(A)\bg_i\bbf_i,\quad \langle \bg_i,\bg_j\rangle=\langle \bbf_i,\bbf_j\rangle=\delta_{ij}, i,j\in\N.\\
&&A_n=\sum_{i=1}\sigma_i(A_n)\bg_{i,n}\bbf_{i,n},\quad \langle \bg_{i,n},\bg_{j,n}\rangle=\langle \bbf_{i,n},\bbf_{j,n}\rangle=\delta_{ij}, i,j\in\N.
\end{eqnarray*}
Then $|\sigma_i(A_n)-\sigma_i(A)|\le \|A_n-A\|$ for each $i,n\in\N$.
Assume that $\sigma_i(A)>0$.  Then there exists $p,q\in\N$, $p\le i \le q$ such that
$\sigma_{p-1}(A)> \sigma_p(A)=\cdots=\sigma_q(A)>\sigma_{q+1}(A)\ge 0$.
Denote by
\begin{eqnarray*}
\bU_{p,q}=\mathrm{span}(\bg_{p},\ldots,\bg_q), \quad \bV_{p,q}=\mathrm{span}(\bbf_{p},\ldots,\bbf_q).
\end{eqnarray*}
 Then there exists $N_{p}=\cdots =N_q\in\N$ such that $\sigma_{p-1}(A_n)>\sigma_p(A_n)$ and $\sigma_{q}(A_n)>\sigma_{q+1}(A_n)$ for $n>N_p$. For $n>N_p$ denote
 \begin{eqnarray*}
\bU_{p,q,n}=\mathrm{span}(\bg_{p,n},\ldots,\bg_{q,n}), \quad \bV_{p,q,n}=\mathrm{span}(\bbf_{p,n},\ldots,\bbf_{q,n}).
\end{eqnarray*}
Then
\begin{eqnarray*}
\lim_{n\to\infty} \|P(\bU_{p,q,n})-P(\bU_{p,q})\|_1=0, \quad \lim_{n\to\infty} \|P(\bV_{p,q,n})-P(\bV_{p,q})\|_1=0.
\end{eqnarray*}
More precisely: Denote by $\bW_{p,q}\subset T_2(\cH)$ the $q-p+1$ subspace spanned by an orthonormal basis $\bg_p\bbf_p^\vee, \ldots,\bg_q\bbf_q^\vee$, and
by $\bW_{p,q,n}\subset T_2(\cH)$ the $q-p+1$ subspace spanned by an orthonormal basis
$$\bg_{p,n}\bbf_{p,n}^\vee, \ldots,\bg_{q,n}\bbf_{q,n}^\vee$$ for $n>N_p$.  Then
$$\lim_{n\to\infty}\|P(\bW_{p,q,n})-P(\bW_{p,q})\|_1=0.$$
\end{enumerate}
\end{lemma}
\begin{proof}  (1) Let $A$ have a singular value decomposition as in (2).
Assume that $B=\sum_{j=1}^\infty \sigma_i(B)\bu_i\bv_i^\vee$ be a singular value decomposition of $B$.  Define
\begin{eqnarray*}
A_m=\sum_{i=1}^m \sigma_i(A)\bg_i\bbf_i^\vee, \;\sigma_1(A)\ge\cdots\ge\sigma_l(A)>0,\;\langle\bg_i,\bg_j\rangle=\langle\bbf_i,\bbf_j\rangle=\delta_{ij}, i.j\in[m].\\
B_m=\sum_{j=1}^m \sigma_i(B)\bu_i\bv_i^\vee, \;\sigma_1(B)\ge\cdots\ge\sigma_m(B)>0,\;\langle\bu_i,\bu_j\rangle=\langle\bv_i,\bv_j\rangle=\delta_{ij}, i.j\in[l].
\end{eqnarray*}
Define
\begin{eqnarray*}
\bX_m=\textrm{span}(\bg_1,\ldots,\bg_m,\bu_1,\ldots,\bu_m), \; \bY_m=\textrm{span}(\bbf_1,\ldots,\bbf_m,\bv_1,\ldots,\bv_m).
\end{eqnarray*}
Then $A_m,B_m:\bY_m\to \bX_m$ and $A_m^\vee, B_m^\vee: \bX_m\to \bY_m$.  Thus we can  view $A_m$ and $B_m$ as  $M_m\times N_m$ complex values matrices $C_m$ and $D_m$ respectively,
 where $M_m=\dim X_m, N_m=\dim Y_m$. The positive singular values of $C_m$ and $D_m$ are identical with the positive singular values of $A_m$ and $B_m$ respectively.   Also,
 
 \begin{eqnarray*}
 \|A_m-B_m\|_2^2&=&\tr (C_m-D_m)(C_m^*-D_m^*)\\
 &=&\tr C_m C_m^*+\tr D_mD_m^* -2\Re \tr C_mD_m^*\\
&=& (\sum_{i=1}^{m} (\sigma_i^2(A_m)+\sigma_i^2(B_m))-2\Re \tr C_mD_m^*.
\end{eqnarray*}
Recall von Neumann inequality \cite[Theorem 4.11.8]{Frb16}:
$$\Re \tr C_mD_m^*\le \sum_{i=1}^M \sigma_i(C_m)\sigma_i(D_m).$$
  This shows
\begin{eqnarray*}
\sqrt{\sum_{i=1}^m (\sigma_i(A)-\sigma_i(B))^2}&\le& \|A_m-B_m\|_2\\
&=&\|(A_m-A)+(B-B_m)+(A-B)\|_2\\
&\le& \|A_m-A\|_2+\|B-B_m\|_2+\|A-B)\|_2.
\end{eqnarray*}
Let $m\to\infty$ to deduce (1).

(2) We first claim that $\lim _{n\to \infty} \|A_n^\vee A_n  -  A^\vee A\|_1=0$.
This follows from Lemma \ref{weaknormconv1}.  First note that $F_n=A_n^\vee A_n, F=A^\vee A\in \rS_{1,+}(\cH)$.  Observe next that $F_n\overset{w.o.t.}{\to}F$.  Indeed,
for a given $\bx,\by\in \cH$ we have
\begin{eqnarray*}
&&\lim_{n\to\infty} \|A_n\bx-A\bx\|=0, \; \lim_{n\to\infty} \|A_n\by-A\by\|=0\Rightarrow\\ &&\lim_{n\to\infty}\langle A_n^\vee A_n \bx,\by\rangle=\lim_{n\to\infty}\langle  A_n \bx,A_n\by\rangle=\langle A\bx,A\by\rangle=\langle A^\vee A\bx,\by\rangle.
\end{eqnarray*}
Clearly $\tr A_n^\vee A_n=\|A_n\|_2^2, \tr A^\vee A=\|A\|^2_2 $.  As $\lim_{n\to\infty} \|A_n\|_2=\|A\|_2$ we deduce that $\lim_{n\to\infty}\tr F_n=\tr F$.

Next observe that the spectral decomposition of $F_n$ and $F$ are:
\begin{eqnarray*}
&&F_n=\sum_{i=1}^\infty \sigma_i(A_n)^2 \bbf_{i,n}\bbf^\vee_{i,n}, \quad n\in\N,\\
&&F=\sum_{i=1}^\infty \sigma_i(A)^2 \bbf_{i}\bbf^\vee_{i}.
\end{eqnarray*}

As
\begin{eqnarray*}
|\sigma_i(A_n)-\sigma_i(A)|\le (\sum_{i=1}^{\infty}(\sigma_i(A_n)-\sigma_i(A))^2)^{1/2}\le \|A_n-A\|_2
\end{eqnarray*}
It follows that $\lim_{n\to\infty}\sigma_i(A_n)=\sigma_i(A)$ for each $i\in\N$.  Assume that $\sigma_q(A)>\sigma_{q+1}(A)$.  Then there exists $N_q$ such that for $n> N_q$ one has the inequalities:
\begin{eqnarray*}
\sigma_q(A_n)>(\sigma_q(A)+\sigma_{q+1}(A))/2>\sigma_{q+1}(A).
\end{eqnarray*}

Let $\bV_q, \bV_{q,n}$ be the projection on the subspace spanned by  $\bbf_1,\ldots,\bbf_q$ and by $\bbf_{1,n},\ldots,\bbf_{q,n}$ respectively.  Set
\begin{eqnarray*}
\bbf_{i,n,q}=\sum_{j=1}^q \langle \bbf_{i,n},\bbf_j\rangle \bbf_j, \quad i\in[q], n\in\N.
\end{eqnarray*}
Assume first the simplest case where $q=1$: $\sigma_1(A)>\sigma_2(A)$.
Recall that  $\sigma_1^2(A_n)=\sigma_1(F_n)=\langle F_n\bbf_{1,n}, \bbf_n\rangle$.
Next observe the inequality
\begin{eqnarray*}
&&|\langle (F_n-F)\bbf_{1,n}, \bbf_{1,n}\rangle|=|\tr(A_n -A)(\bbf_{1,n}\bbf_{1,n}^\vee)|\le \\
&&\|F_n-F\|_1\| (\bbf_{1,n}\bbf_{1,n}^\vee)\|= \|F_n-F\|_1.
\end{eqnarray*}
Here $ \| (\bbf_{1,n}\bbf_{1,n}^\vee)\|$ is the operator norm of $\bbf_{1,n}\bbf_{1,n}^\vee$.  Hence it is equal to $1$.
The maximum principle for $\sigma_1(F)$, the maximum eigenvalue of $F$ yields
\begin{eqnarray*}
\sigma_1(F)\ge \langle F\bbf_{1,n}, \bbf_{1,n}\rangle\ge\sigma_1(F_n)-\|F_n-F\|_1.
\end{eqnarray*}
Recall that $\lim_{n\to\infty}\sigma_1(F_n)=\sigma_1(F)$ and $\lim_{n\to\infty} \|F_n-F\|_1=0$.  Thus $\lim_{n\to\infty}  \langle F\bbf_{1,n}, \bbf_{1,n}\rangle=\sigma_1(F)$.
Observe next
\begin{eqnarray*}
 \langle F\bbf_{1,n}, \bbf_{1,n}\rangle&=&\sum_{i=1}^\infty \sigma_i^2(A)|\langle \bbf_{1,n},\bbf_{1}\rangle|^2\\
&\le& \sigma_1(A)^2 |\langle \bbf_{1,n},\bbf_{1}\rangle|^2+\sigma_2(A)^2 (1-|\langle \bbf_{1,n},\bbf_{1}\rangle|^2).
\end{eqnarray*}
Hence $\lim_{n\to\infty} |\langle \bbf_{1,n},\bbf_{1}\rangle|=1$, which is equivalent to $\lim_{n\to\infty}\|P(\bV_{1,n})-P(\bV_1)\|_1=0$.  (This is equivalent to that we could choose the phase of $\bbf_{1,n}$ so that $\lim_{n\to\infty}\|\bbf_{1,n}-\bbf_1\|=0$.)

In the  general case one needs to pass to the $q$-wedge product $\wedge^q A_n, \wedge^q A$ as in \cite{FP04}.
\end{proof}

\end{document}